\documentclass[11pt]{amsart}
\usepackage{amsmath}
\usepackage{amsthm}
\usepackage{amsbsy}
\usepackage{amsopn}
\usepackage{amsmath,amscd}
\usepackage{amssymb}
\usepackage{amsfonts}
\usepackage{multirow}
\usepackage{mathrsfs} 
\usepackage[official ]{eurosym}

\newcommand{\RNum}[1]{\uppercase\expandafter{\romannumeral #1\relax}}

\newcommand\cA{\mathcal{A}}

\newcommand\sA{\mathscr{A}}

\newcommand\sT{\mathscr{T}}

\newcommand{\risom}{\stackrel{\sim}{\to}}

\newcommand{\cont}{\mathit{cont}}

 \usepackage[all]{xy}
                \usepackage{amsthm,amsfonts,amsmath,amscd,amssymb,epsfig,verbatim,
}

\usepackage{graphicx}
\usepackage{epstopdf}
\date{\today} 
 
\usepackage{graphicx}

\usepackage[dvipsnames,svgnames,x11names,hyperref]{xcolor}
\usepackage{url,graphicx,verbatim,amssymb,enumerate,stmaryrd}
\usepackage[pagebackref,colorlinks,citecolor=Bittersweet,linkcolor=Bittersweet,urlcolor=Bittersweet,filecolor=Bittersweet]{hyperref}
\usepackage{tikz} 
\usetikzlibrary{arrows,decorations.pathmorphing,backgrounds,fit,positioning,shapes.symbols,chains,calc}
\usepackage[all]{xy}
\xyoption{matrix}
\xyoption{arrow}
\usepackage[hmargin=3cm,vmargin=3cm]{geometry}
\usepackage{setspace,kantlipsum}
\tikzset{help lines/.style={step=#1cm,very thin, color=gray},
help lines/.default=.5} % draws a grid spaced #1 cm

\usepackage{booktabs}

\theoremstyle{plain}
\newtheorem{theorem}{Theorem}[section]

\newtheorem{lemma}[theorem]{Lemma}
\newtheorem{proposition}[theorem]{Proposition}
\newtheorem{prop}[theorem]{Proposition}

\newtheorem{cor}[theorem]{Corollary}

\theoremstyle{definition}
\newtheorem{definition}[theorem]{Definition}
\newtheorem{inductive step}[theorem]{Inductive step}

\newtheorem{inductive lemma}[theorem]{Inductive Lemma}
\theoremstyle{remark}

\newtheorem{remark}[theorem]{Remark}
\newtheorem*{remark*}{Remark}

\newtheorem*{example*}{Example}

\newcommand{\wh}{\widehat}
\newcommand{\ol}{\overline}

\newcommand{\p}{\partial}

\newcommand{\om}{\omega}

\newcommand{\eps}{\varepsilon}

\newcommand{\R}{{\mathbb{R}}}

\newcommand{\Aut}{{\rm Aut}}

\newcommand{\sgn}{{\rm sgn\,}}

\newcommand{\codim}{{\rm codim}}

\newcommand{\Span}{\mathrm{Span}}

\newcommand{\cT}{{\widehat{\sT}}}
\newcommand{\ssT}{\mathrm{t}}
\def\Op{{\mathcal O}{\it p}\,}
\numberwithin{figure}{section}

\newcommand{\symp}{\mathit{symp}}% Added

 \newcommand{\Arb}{\mathrm{Arb}} %Added

\title{Arboreal models and their stability}

\author{Daniel  \'Alvarez-Gavela}
\address{Department of Mathematics \\ Massachusetts Institute of Technology \\ Cambridge, MA, 02139}
\email{dgavela@mit.edu}
\thanks{DA was partially supported by NSF grant DMS-1638352 and the Simons Foundation}

\author{Yakov Eliashberg }
\address{Department of Mathematics\\Stanford University \\ Stanford, CA 94305}
\email{eliash@stanford.edu}
\thanks{YE was partially supported by NSF grant DMS-1807270. }

\author{David Nadler}
\address{Department of Mathematics\\University of California, Berkeley\\Berkeley, CA  94720-3840}
\email{nadler@math.berkeley.edu}
\thanks{DN was partially supported by NSF grant DMS-1802373.}

\begin{document}
\begin{abstract}
 This is the first in a series of papers~\cite{AGEN19, AGEN20b, AGEN21} by the authors on the arborealization program. The  main goal of the paper is the  proof of uniqueness  of arboreal models, defined   as the closure of the class of smooth germs of Lagrangian submanifolds under the operation of taking iterated transverse Liouville cones. The parametric version of the stability result implies that the space of germs of symplectomorphisms that preserve a canonical model is weakly homotopy equivalent to the space of automorphisms of the corresponding signed rooted tree. Hence the local symplectic topology around a canonical model reduces to combinatorics, even parametrically.
 \end{abstract}

\maketitle

 \onehalfspacing
 \tableofcontents
  \section{Introduction}\label{sec:intro}
 
 \subsection{Main results}

 This is the first in a series of papers~\cite{AGEN19, AGEN20b, AGEN21} by the authors on the arborealization program.

 The initial goal of this program is to determine when a Weinstein manifold can be deformed to have an arboreal skeleton, i.e.~a skeleton which is a stratified Lagrangian with arboreal singularities. The main results of \cite{AGEN20b} show this can be achieved for polarized Weinstein manifolds, and moreover,   the resulting arboreal space determines the ambient Weinstein manifold. The current paper provides essential local results underlying these global developments, as will be discussed below.

%  establish the uniqueness properties of such an arboreal skeleton. Beyond this, the implications of the existence and uniqueness of an arboreal skeleton for the study of Weinstein manifolds will also be pursued in future work of the authors.
 
 The class of arboreal singularities was first defined by the third author in the paper ~\cite{N13} for abstract $n$-dimensional complexes (without any smooth structure).  Their Lagrangian and Legendrian realizations were also exhibited  in  ~\cite{N13} by means of explicit models. In ~\cite{St18}  and \cite{E18} these models were further decorated by signs.
 It is important to point out that the definition in ~\cite{N13} fixes  only the {\em homeomorphism}, and not {\em diffeomorphism}   type of the singularity.   While this is sufficient for many applications, for example calculating  of some invariants,     the {\em homeomorphism} type of an arboreal skeleta does not determine in general the symplectomorphism type of the ambient manifold,  even if the skeleton is smooth (e.g. see ~\cite{Ab12}).  On the other hand, it is not even possible to discuss the  problem of whether the symplectic  topology of the Weinstein manifold is determined by its  skeleton   before establishing canonical up to symplectomorphism local models of singularities. 

We prove in the current paper   that  the  singularities in the class of {\em  signed arboreal  Lagrangian and Legendrian singularities} introduced by  Definition \ref{def:arb intr} below     are {\em determined up to ambient symplectomorphism} by their combinatorial type,  and present their canonical models.
The uniqueness problem was not even considered   in the prior papers on this subject.  
 %  Moreover, these  models naturally appear  in the dictionary between polarized Weinstein manifolds and arboreal skeleta established in \cite{AGEN20b}. Due to the central role of these canonical models, we will typically refer to them as arboreal singularities; other members of the more general class  will not play a role in this paper or its sequels.

 We find it  surprising that  there is a solution to the problem of finding such canonical models.  Indeed, we do not know any other sufficiently large classes of Lagrangian singularities which admit a discrete classification up to ambient symplectomorphism.
 Following Definition~\ref{def:arb intr}, given a choice of canonical model,  if we  take its Legendrian lift, apply a contactomorphism taking it into generic position, and then form its Liouville cone, we must once again obtain a canonical model. 
 This is essentially equivalent to the statement that any sufficiently small deformation of a Legendrian model by a contactomorphism of $PT^*\R^n$ can be realized by an element of a smaller subgroup   of contactomorphisms  formed by contact lifts  of diffeomorphisms of $(\R^n,0)$.   
 % The Stability Theorem \ref{thm:unique-signed}, which is our main result, solves this problem. Indeed, its proof is a systematic procedure to rectify the relevant class of contactomorphisms. 
%%\ki{}{ One can contemplate generalizations to other Lagrangian singularities and quickly see that what we achieve is very special to arboreal singularities.}
 
 To discuss this in more detail,  we first introduce some auxiliary notions. A closed 
subset of a symplectic  or  contact  manifold is called {\em isotropic} 
 if is stratified by  isotropic submanifolds. It is called Lagrangian or Legendrian if it is isotropic and purely of the maximal possible dimension. The germ at the origin of a  locally simply-connected isotropic subset
   $L\subset T^*\R^n$ of the cotangent bundle with its standard Liouville structure $\lambda = pdq$ admits a unique   lift to an isotropic germ at the origin $\wh L\subset J^1\R^n=T^*\R^n\times\R$ of the 1-jet bundle. 
   Given an isotropic subset $\Lambda\subset S^*\R^n$ of the cosphere bundle, its Liouville cone $C(\Lambda) \subset T^*\R$, i.e.~the closure of its saturation by trajectories of the Liouville vector field $Z = p\frac{\p}{\p p}$,
   is an isotropic subset.
  %In analogy with Morse functions,  one might hope the skeleton $\Skel(W, \lambda)$  of a suitably ``generic" Weinstein domain $(W, \lambda, \phi)$ might have simple singularities. 
    
    \begin{definition}\label{def:arb intr}
  {\it Arboreal  Lagrangian (resp.~Legendrian) singularities}  form the smallest class $\Arb^{\symp}_n$  (resp.~$\Arb^{\cont}_n$) of germs of    closed isotropic subsets in  $2n$-dimensional symplectic (resp.~$(2n+1)$-dimensional   contact) manifolds such that the following properties are satisfied:
    \begin{enumerate}
    \item(Invariance) $\Arb^\symp_n$ is invariant with respect to symplectomorphisms and $\Arb^\cont_n$ is invariant with respect to contactomorphisms.
    \item (Base case) $\Arb^\symp_0$ contains $pt = \R^0 \subset T^*\R^0 = pt$.
    
    \item (Stabilizations) If  $L \subset (X, \omega)$ is  in $\Arb^\symp_n$,  then  the  product $L \times \R \subset (X \times T^*\R, \omega + dp\wedge dq)$
    is in $Arb^\symp_{n+1}$. 
    \item (Legendrian lifts) If $L\subset T^*\R^n$ is in $\Arb^\symp_n$, then its Legendrian lift $\wh L\subset J^1\R^n$ is in $\Arb^\cont_n$.
    \item (Liouville cones) Let $\Lambda_1,\dots, \Lambda_k \subset  S^*\R^n$ be a  finite disjoint  union of arboreal Legendrian germs from $Arb^\cont_{n-1}$ centered at points $z_1,\dots, z_k \in S^* \R^n$. Let $\pi:S^*\R^n\to \R^n$ be   the front projection. Suppose
    \begin{itemize}\item[-] $\pi(z_1)=\dots=\pi(z_k)$;
    \item[-] For any $i$, and smooth submanifold $Y \subset \Lambda_i$, the restriction  $\pi|_Y:Y\to\R^n$ is an embedding (or equivalently, an immersion, since we only consider germs).
     \item[-] For any distinct $i_1, \ldots, i_\ell$, and any  smooth submanifolds $Y_{i_1}\subset \Lambda_{i_1}, \dots,Y_{i_\ell}\subset \Lambda_{i_\ell}$,  the restriction 
    $\pi|_{Y_{i_1}\cup \dots\cup Y_{i_\ell}}: Y_{i_1}\cup \dots\cup Y_{i_\ell} \to \R^n$ is self-transverse.
    \end{itemize} 
    Then the union $\R^n \cup C(\Lambda_1)\cup\dots\cup C(\Lambda_k)$ of the Liouville cones  with the zero-section form an arboreal Lagrangian germ from    $\Arb^\symp_n$.
    \end{enumerate}

 With the above classes defined, we can also allow boundary by additionally taking the product  $L \times \R_{\geq 0} \subset (X \times T^*\R, \omega + dp\land  dq)$
    for any arboreal Lagrangian $L \subset (X, \omega)$, and similarly  for arboreal Legendrians.  

    \end{definition}
%  
% \begin{figure}[h]
%\includegraphics[scale=0.5]{A1Stan.pdf}
%\caption{The simplest non-smooth arboreal singularity is the zero section union the positive conormal of a smooth co-oriented hypersurface.}
%\label{fig:A1stab}
%\end{figure}
%%
% \begin{figure}[h]
%\includegraphics[scale=0.42]{NotAroot}
%\caption{This arboreal singularity consists of the zero section union the positive conormal of two smoothly intersecting co-oriented hypersurfaces.}
%\label{fig:NotAroot}
%\end{figure}
%%
%%    
% \begin{figure}[h]
%\includegraphics[scale=0.5]{A2}
%\caption{This arboreal singularity consists of the zero section union the positive conormal of a singular co-oriented hypersurface (namely the front of the arboreal singularity of Figure \ref{fig:A1stab}).}
%\label{fig:A2}
%\end{figure}
    
In Section \ref{sec:stab}, we prove the Stability Theorem \ref{thm:unique-signed} for arboreal models as characterized by
Definition~\ref{def:arb intr}. Its main application is the following: for fixed dimension $n$, up to ambient symplectomorphism or contactomorphism, Definition~\ref{def:arb intr} produces only finitely many local models. 

More precisely, to each member of  the class $\Arb^{\symp}_n$, one can assign a signed rooted tree $\cT= (T, \rho, \eps)$ 
    with at most $n+1$ vertices.
    Here $T$ is a  finite acyclic graph, $\rho$
      is a distinguished root vertex, and  $\eps$ 
           is a sign function on the edges of $T$ not adjacent to $\rho$. This discrete data completely determines the germ:
         
         \begin{theorem}\label{thm: main thm}
         If two arboreal Lagrangian singularities $L \subset (X, \omega)$, $L' \subset (X', \omega')$ of the class $\Arb^{\symp}_n$
         have the same dimension and  signed rooted tree $\cT$, then there is (the germ of) a symplectomorphism 
         $(X, \omega) \simeq (X', \omega')$ identifying $L$ and $L'$.

         \end{theorem}
         
         Similarly, each member of the class  $\Arb^{\cont}_n$ is determined by an associated signed rooted tree $\cT= (T, \rho, \eps)$ 
    with at most $n+1$ vertices.

         As a representative for each signed rooted tree $\cT$, one may take the local model $L_\cT \subset T^*\R^n$ detailed in Section \ref{sec: arboreal}, where
          $n=|n(\cT)|$ is one less than the number of vertices in the tree. 
          The model  $L_\cT \subset T^*\R^n$  is given as the positive conormal to an explicit front $H_\cT \subset \R^n$ defined by elementary equations. Our exposition in Section \ref{sec: arboreal} is self-contained and the reader need not be familiar with \cite{N13} or \cite{St18}. In particular, we emphasize the inductive nature of the constructions, and  reformulate signs to match inductive arguments to come. 
          
          In Section \ref{sec:stab}, we also establish a parametric version of the Stability Theorem  \ref{thm:unique-signed}, whose main consequence can be formulated as follows:
         
         \begin{theorem}\label{thm: contr auts}
Fix a signed rooted tree $\cT = (T, \rho, \eps)$, set $n = |n(\cT)|$ and consider  the arboreal $\cT$-front  $H_\cT \subset \R^n$. Let $D(\R^n, H_\cT)$ be the group of germs at $0$  of diffeomorphisms of $\R^n$ preserving $H_\cT$ as a front, i.e.~ as a subset along with its coorientation.

Then the fibers of the natural map $D(\R^n, H_\cT) \to \Aut(\cT)$ are weakly contractible.
\end{theorem}

Hence, the local symplectic topology of an arboreal singularity is completely characterized by the combinatorics of the underlying signed rooted tree, even parametrically. 

We conclude this introduction by briefly explaining the role of Theorem \ref{thm: main thm} within the global results of the arborealization program. It was shown in \cite{N15} that singularities of  Whitney stratified Lagrangians can always be locally deformed  to arboreal Lagrangians in a non-characteristic fashion, i.e.~without changing their microlocal invariants. The question of whether a global theory exists at the level of Weinstein structures is more subtle.
In two dimensions the story is classical: generic ribbon graphs provide arboreal skeleta. In four dimensions, Starkston proved in ~\cite{St18} that
 arboreal skeleta always exist in the Weinstein homotopy class of any Weinstein domain.

 In the sequel \cite{AGEN20b}, we show any polarized Weinstein manifold,
 i.e.~a Weinstein manifold with a global field of Lagrangian planes in its tangent bundle, 
  can be deformed to have an arboreal skeleton. More specifically, the arboreal singularities that arise are positive in the sense that they are indexed by signed rooted trees with all positive signs  $+1$, and  conversely, any Weinstein manifold with a positive arboreal skeleton comes with a canonical (homotopy class of) polarization.
  
 Now, the arguments of  \cite{AGEN20b} produce skeleta with singularities satisfying the characterization of Definition~\ref{def:arb intr}. So without the uniqueness of Theorem~\ref{thm: main thm}, we would still be left to study the possible moduli of such singularities: it could happen that two skeleta built with the same combinatorics lead to different Weinstein manifolds. The uniqueness of Theorem~\ref{thm: main thm} guarantees this is not the case: there is no moduli of the singularities arising, and indeed  their geometry is uniquely specified by the combinatorics. (In fact, the situation is even better: thanks to the arboreal Darboux-Weinstein theorem proved in  \cite{AGEN20b}, any symplectic thickenings  of the skeleton that induce equivalent orientation structures, a further combinatorial decoration on the skeleton, are themselves equivalent.) Thus pairing the results of the current paper  with those of \cite{AGEN20b} one is able to express polarized Weinstein manifolds in  combinatorial terms. In a forthcoming paper \cite{AGEN21} we will classify all bifurcations ($=$ "Reidemeister moves") relating arboreal skeleta of two    polarized Weinstein manifolds  related by a polarized Weinstein homotopy, thus  reducing the problem of  classification of (polarized) Weinstein structures  up to deformational equivalence to the problem of classification of arboreal complexes up to diffeomorphism and Reidemeister moves.

 \subsection{Acknowledgements} 
 We are very grateful to Laura Starkston who collaborated with us on the initial stages of this project. 
% We are also grateful to John Pardon who helped us to crystalize the notion of an (n, n ? 1)-polarization and to Søren Galatius who explained to us how to compute obstructions to its existence.
The first author is grateful for the great working environment he enjoyed at Princeton University and the Institute for Advanced Study, as well as for the hospitality of the Centre de recherches math\'ematiques of Montreal. 
The second author thanks RIMS Kyoto and ITS ETH Zurich for their hospitality. The third author thanks MSRI for its hospitality.
Finally, we are very grateful for the support of the American Institute of Mathematics, which hosted a workshop on the arborealization program in 2018 from which this project has greatly benefited.

\section{Arboreal models}\label{sec: arboreal}

\subsection{Quadratic fronts}\label{s: quad fronts}

Before we present the local models for arboreal singularities, we introduce the quadratic fronts out of which the models will be built and discuss some of their basic properties.

\subsubsection{Basic constructions}

For $i\geq 0$, define functions $h_i:\R^i \to \R$ by the inductive formula
$$
\xymatrix{
h_0 = 0
&
h_i = h_i(x_1, \ldots, x_i) = x_1 - h_{i-1}(x_2, \ldots, x_{i})^2
}
$$
For example, for small $i$, we have
$$
\xymatrix{
h_1(x_1) = x_1
&
h_2(x_1, x_2) = x_1 - x_2^2
&
h_3(x_1, x_2, x_3) =x_1 - (x_2 - x_3^2)^2
}
$$
%For $n\geq i$, we often view $g_i$ as a function  $g_i:\R^n \to \R$ independent of the coordinates $x_{i+1}, \ldots, x_n$.
Fix $n\geq 0$. For $i = 0, \ldots, n$,  define smooth graphical hypersurfaces 
$$
\xymatrix{
 {}^n\Gamma_i = \{ x_{0} =  h_i^2 \}  \subset \R^{n+1}
}
$$
equipped with the graphical coorientation,
and consider their union
$$
\xymatrix{
 {}^n\Gamma = \bigcup_{i=0}^n   {}^n \Gamma_i
}
$$
%We equip ${}^n\Gamma_i$ with the graphical coorientation, and regard  ${}^n\Gamma$ as multi-cooriented by the
%union of these coorientations.
Note the elementary identities
$$\xymatrix{
{}^n\Gamma_i =  {}^i\Gamma_i \times \R^{n-i} & i = 0, \ldots, n
}
$$
$$
\xymatrix{
 {}^n \Gamma_i  \cap  {}^n\Gamma_0  =  {}^{n-1} \Gamma_{i -1} 
 & i = 1, \ldots, n
}
$$

          \begin{figure}[h]
\includegraphics[scale=0.53]{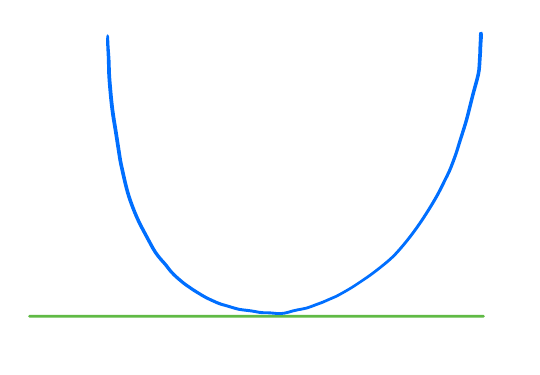}
\caption{The hypersurfaces ${}^1\Gamma_0$ (green) and ${}^1\Gamma_1$ (blue)}
\label{fig:Modelone}
\end{figure}

    \begin{figure}[h]
\includegraphics[scale=0.5]{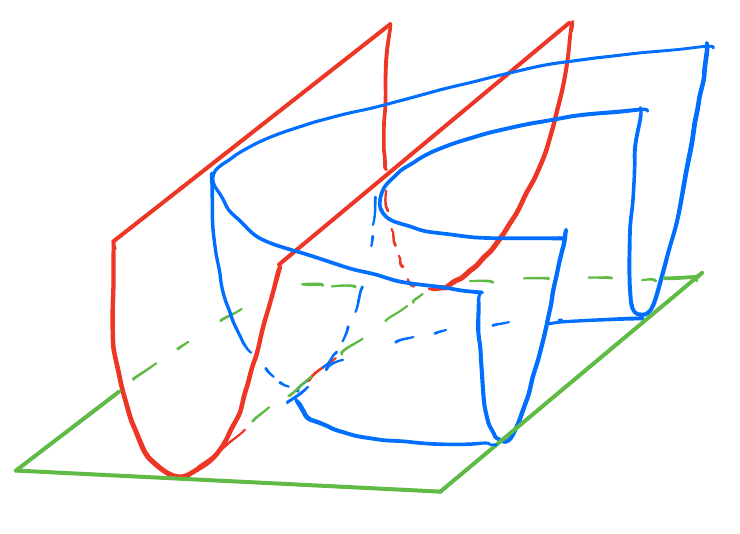}
\caption{The hypersurfaces ${}^2\Gamma_0$ (green), ${}^2\Gamma_1$ (red) and ${}^2\Gamma_2$ (blue).}
\label{fig:canonicals}
\end{figure}

Let $T^*\R^n$ denote the cotangent bundle with canonical 1-form $pdx = \sum_{i = 1}^np_idx_i$ where
$p=(p_1, \ldots, p_n)$ are dual coordinates to $x = (x_1, \ldots, x_n)$. 
Let $J^1\R^n = \R \times T^*\R^n$ denote the 1-jet bundle with contact form $dx_0 + pdx = dx_0 + \sum_{i = 1}^np_idx_i$.

%Let $S^*\R^{n+1} = (T^*\R^{n+1} \setminus 0)/\R_{>0}$ denote the cosphere bundle with its canonical cooriented contact structure. If $p_0$ is the dual coordinate to $x_0$, 
% then we have the contact level-set 
% $\{p_0 = \pm 1\}\subset T^*\R^{n+1}$ with contact form $\pm dx_0 + pdx = \pm dx_0 + \sum_{i = 1}^np_idx_i$.
% We have
%  an evident open cooriented contact embedding  $\{p_0 = \pm 1\} \subset S^*\R^{n+1}$, and cooriented contactomorphisms
%  $\{p_0 = \pm 1\} \simeq J^1 \R^n$ with $x_0 \mapsto \pm x_0$.

Given a function $f:\R^n \to \R$ with graph $\Gamma_f = \{x_0 = f(x)\} \subset \R \times \R^n$, we have the conormal Lagrangian of the graph 
$L_{\Gamma_f} = \{ x_0 = f(x), \, p_i= - p_0  \partial f / \partial x_i  \} \subset T^*\R^{n+1}$, and
the conormal Legendrian of the graph $\Lambda_{\Gamma_f}  = \{ x_0 = f(x), \, p_0 = 1, \, p_i = - \partial f / \partial x_i\} \subset J^1 \R^n$.

%There is a unique Legendrian $\Lambda_f \subset J^1\R^{n+1}$ projecting diffeomorphically to the front  $\Gamma_f$ given by $\Lambda_f = \{ x_0 = f(x), p_0 = 1, pdx =  df\}$.

%There are unique Legendrians $\pm \Lambda_f \subset  \{p_0 = \pm 1\} \subset S^*\R^{n+1}$ projecting diffeomorphically to the front  $\Gamma_f$ given by $\pm \Lambda_f = \{ x_0 = f(x), p_0 = \pm 1, pdx = \mp df\}$.
%

For $i = 0$, let ${}^{n} L_0 =   \R^{n} \subset T^*\R^{n}$ denote the zero-section.
For $i = 1, \ldots, n$,  introduce the  conormal Lagrangian
$$
\xymatrix{
 {}^{n} L_{i} =  L_{{}^{n-1} \Gamma_{i-1}} \subset T^* \R^{n}
}
$$
of the graph ${}^{n-1} \Gamma_{i-1} \subset \R^n$, and consider their union
$$
\xymatrix{
{}^{n} L = \bigcup_{i = 0}^n {}^n L_i
}
$$

%Regard ${}^n\Gamma_i$ as cooriented by the positive $\partial_{x_0}$ direction
Similarly, for $i = 0, \ldots, n$, introduce the conormal Legendrian 
$$
\xymatrix{
 {}^{n} \Lambda_i = \Lambda_{{}^{n} \Gamma_{i}} \subset  J^1\R^{n}
}
$$
of the graph ${}^n\Gamma_i \subset \R^{n+1}$, 
and consider their union
$$
\xymatrix{
 {}^n \Lambda = \bigcup_{i =0}^n {}^{n} \Lambda_i
}
$$

Note that the Liouville form vanishes on the conical Lagrangian ${}^nL_i \subset T^*\R^n$, hence its lift to $J^1 \R^n = \R \times T^*\R^n$ with zero primitive is a Legendrian. We have the following compatibility:
%Let $p = (p_1, \ldots, p_n)$ be dual coordinates to $x= (x_1, \ldots, x_n)$, and let $dx_0-\sum\limits_1^np_idx_i$ be  the contact form on $J^1\R^n$.
 
\begin{lemma}\label{lem: tilt}
The contactomorphism
$$
\xymatrix{
S: J^1 \R^n \ar[r] & J^1 \R^n
}
$$
$$
\xymatrix{
S(x_0, x, p) = (x_0 - p_1^2/4, x_1 + p_1/2, x_2, \ldots, x_n, p_1, \ldots, p_n)
}
$$
takes the  Legendrian  $ {}^{n} \Lambda_i $  isomorphically to
the Legendrian
$\{0\} \times {}^n L_i $, and
thus the union $ {}^{n} \Lambda$  isomorphically to the union
$\{0\} \times {}^n L$.

\end{lemma}

\begin{proof}
Set 
$
h_{i, 1} = h_{ i -1}(x_{2}, \ldots, x_i)
$ so that $h_i = x_1 - h_{i,1}^2$.
%Recall ${}^n\Gamma_i \subset \R \times \R^n$ is the graph of $h_i^2 = (x_1 - h_{i,1}^2)^2$ where $h_{i,1}$ is independent of $x_1$. 
Observe 
${}^{n}\Lambda_i \subset J^1 \R^{n}$ is given by the equations
$$
\xymatrix{
x_0 = h_i^2 & pdx = -dh_i^2 =  -2h_i dh_i = -2h_i( dx_1 - 2h_{i,1} dh_{i,1})
}
$$
so in particular $p_1 = -2h_i$ and $\sum_{i = 2}^n p_i dx_i = 4 h_i  h_{i,1} d h_{i, 1}$.

If we write $(\hat x_0, \hat x, p) = S(x_0, x, p)$, for  $(x_0, x, p) \in {}^{n} \Lambda_i$, then we have
$$
\xymatrix{
\hat x_0 = x_0 - p_1^2/4 = \pm(x_0 -h_i^2) = 0   & \hat x_1 = x_1 +p_1/2 = x_1 - h_i = x_1 - (x_1 - h_{i, 1}^2) = h_{i, 1}^2
}
$$

Now it remains to observe  
${}^{n} L_i \subset T^* \R^{n}$ is given by the equations
$$
\xymatrix{
x_1 = h_{i, 1}^2 & \sum_{i = 2}^n p_i dx_i = -p_1 dh^2_{i,1} = -2p_1 h_{i,1} d h_{i, 1} 
}
$$
This completes the proof. \end{proof}

\subsubsection{Distinguished quadrants}\label{sec: quads} We now specify some distinguished quadrants of the ${}^n\Gamma$ which we will use to define our arboreal models. Which of these quadrants are cut out by our sign conventions will become clearer when the arboreal models are introduced.

For $0 \leq j < i \leq n$, set 
$$
\xymatrix{
h_{i, j} := h_{ i - j}(x_{j+1}, \ldots, x_i)
}
$$
so in particular $h_{i, 0} = h_i(x_1, \ldots, x_i)$ and $h_{i, i-1} = h_1(x_i) = x_i$.

For fixed $0 \leq  i \leq n$, consider the collection of functions
$$
\xymatrix{
h_{i, 0}, \ldots, h_{i, i-1}
}
$$
%$$
%\xymatrix{
%h_1(x_i), \ldots,  h_{ i - j}(x_{j+1}, \ldots, x_i) , \dots, h_i(x_1, \ldots, x_i):\R^i \ar[r] &  \R
%}
%$$
Note the triangular nature of the linear terms of the collection: for all $0 \leq j \leq i-1$, the subcollection
$$
\xymatrix{
h_{i, j} - x_{j+1}, h_{i, j+1}, \ldots, h_{i, i-1} 
}
$$ 
is independent of $x_{j+1}$. Thus the level sets of the collection are mutually transverse.

Fix once and for all a list of signs $\delta = (\delta_0, \delta_1, \ldots, \delta_n)$, $\delta_i \in \{\pm 1\}$.  
Define the domain quadrant $ {}^n  Q^\delta_i \subset \R^n$ to be cut out by the inequalities
$$
\xymatrix{
\delta_1 h_{i, 0} \leq 0, \ldots, \delta_i h_{i, i-1} \leq 0
}
$$
By the transversality noted above, ${}^n  Q^\delta_i$ is a submanifold with corners diffeomorphic to $\R_{\geq 0}^i \times \R^{n-i}$.
Its codimension one boundary faces are  given by the vanishing of one of the functions~$h_{i, j}$.

Note  ${}^n Q^\delta_i $ only depends on the truncated list $\delta_1, \ldots, \delta_i$. In particular, it is independent of 
$\delta_0$ which will enter the constructions next.

Define the cooriented hypersurface  ${}^n  \Gamma_i|_\delta \subset \R^{n+1}$ to be
the restricted signed graph
$$
\xymatrix{
{}^n  \Gamma_i|_\delta = \{x_0 = \delta_0 h_i^2\}|_{{}^n Q^\delta_i} %\{ x_{0} = h_i(x_1, \ldots, x_i)^2 \, |\,  (x_1, \ldots, x_i) \in Q_i \}
}
$$
with the graphical coorientation.

Thus ${}^n \Gamma_i|_\delta$ is cut out by the equations
$$
\xymatrix{
x_0 =  \delta_0 h_i^2 ,&
\delta_1 h_{i, 0} \leq 0, \ldots, \delta_i h_{i, i-1} \leq 0
}
$$
 Since ${}^n  \Gamma_i|_\delta$ is graphical over ${}^n  Q^\delta_i$,  it is also a submanifold with corners diffeomorphic to $\R_{\geq 0}^i \times \R^{n-i}$.  Likewise, its codimension one boundary faces are given by the vanishing of one of the functions $h_{i, j}$. 
 %We equip ${}^n  P^\delta_i$ with the graphical  coorientation.  

Consider as well the union 
$$
\xymatrix{
{}^n  \Gamma|_\delta = \bigcup_{i = 0}^n {}^n \Gamma_i|_\delta 
}
$$

%We will see  in Corollary~\ref{cor: embed}
%that ${}^n  P^\delta_i$, ${}^n  P^\delta_j$ intersect precisely along their tangency locus $T({}^n P^\delta_i, {}^n P^\delta_j)$,
%hence ${}^n P^\delta$ is not only multi-cooriented but  in fact cooriented.

 %%%%
%\subsubsection{Arboreal models}\label{s: can models}

 \begin{remark}
 Note that
$$
\xymatrix{
 {}^n \Gamma_i = \bigcup_{\delta, \delta_0 = 1} {}^n \Gamma_i|_\delta &  {}^n \Gamma = \bigcup_{\delta,  \delta_0 = 1} {}^n \Gamma|_\delta
}
$$
since $x\in  {}^n \Gamma_i$ implies $x\in  {}^n \Gamma_i|_\delta$ where for $1\leq j \leq i$, we set
$\delta_j = -\sgn(h_{i, j}(x))$, when $h_{i, j}(x) \not =  0$, and choose it arbitrarily otherwise.
 \end{remark}

\begin{remark}\label{rem: last sign}
Note if we set $\delta' = (\delta_0, \ldots, \delta_{n-1}, -\delta_n)$, then the map $ \R^{n+1} \to \R^{n+1}$, $(x_0, \ldots, x_{n-1}, x_n) \mapsto (x_0, \ldots, x_{n-1}, -x_n)$, takes  ${}^n  \Gamma|_\delta$ isomorphically to ${}^n  \Gamma|_{\delta'}$ as a cooriented hypersurface. 
Thus we could always set $\delta_n = 1$ and not miss any new geometry.  
\end{remark}

Note ${}^n \Gamma_i \cap \{ x_0 <  0 \}$, hence also ${}^n \Gamma_i|_\delta \cap \{ \delta_0 x_0 <  0 \}$, 
is empty since ${}^n \Gamma_i$ is the graph of $h_i^2\geq 0$.
%The following is a variation on Lemma~\ref{lem: tilt}.

 \begin{lemma}\label{lem: induction for fronts}
 Fix $\delta = (\delta_0, \ldots, \delta_n)$,  and set $\delta' = (\delta_0\delta_1, \delta_2, \ldots, \delta_n)$. The homeomorphism 
 $$
 \xymatrix{
 s:\delta_0\R_{\geq 0} \times \R^n \ar[r] & \delta_0\R_{\geq 0} \times \R^n 
 }
 $$
 $$
 \xymatrix{
  s (x_0, x_1, x_2,  \ldots, x_n) =  (x_0, \delta_0\delta_1(x_1 +  \delta_1 \sqrt{\delta_0 x_0}),x_2, \ldots, x_n)
 }
 $$
 gives a  cooriented identification
 $$
 \xymatrix{
 s({}^n \Gamma_i|_\delta \cap \{ \delta_0 x_0 \geq 0\} ) =  \delta_0\R_{\geq 0} \times {}^{n-1} \Gamma_{i-1}|_{\delta'}  & 0 < i \leq n
 }
 $$
% and thus  a multi-cooriented identification
% $$
% \xymatrix{
% \tau({}^n P^\delta \cap \{ \delta_0 x_0>0\} ) = \delta_0 \R_{>0} \times {}^{n-1} P^{\delta'} 
% }
% $$
 \end{lemma}
 
\begin{proof}
Recall ${}^n \Gamma_i|_\delta$ is defined by
$$
\xymatrix{
x_0 =  \delta_0 h_i^2 &
\delta_1 h_{i, 0} \leq 0, \ldots, \delta_i h_{i, i-1} \leq 0
}
$$
in particular
$$
\xymatrix{
  x_0 = \delta_0 h_i^2 & \delta_1  h_{i, 0} = \delta_1  h_{i}  \leq 0
}
$$
Note the functions $h_{i, 1}, \ldots, h_{i, i-1}$ are independent of the coordinates $x_0, x_1$.

When $  \delta_0 x_0 \geq 0$ and $\delta_1 h_{i}  \leq 0$, the equation $   x_0 = \delta_0 h_i^2$  is equivalent to $\sqrt{   \delta_0 x_0} = -\delta_1 h_i$. Expanding this in terms of the definitions, we can rewrite this in the form
$$
\xymatrix{
x_1 + \delta_1 \sqrt{ \delta_0  x_0} = h_{i-1}(x_2, \ldots, x_i)^2
}
$$
Thus since $\delta_0' = \delta_0\delta_1$, we see $s$ takes ${}^n \Gamma_i|_\delta \cap \{ \delta_0  x_0 \geq 0\}$  into 
$\delta_0 \R_{ \geq 0} \times \{x_1 = \delta'_0  h_{i-1}^2\}$. 

Moreover,  the additional functions $h_{i, 1}, \ldots, h_{i, i-1}$ cutting out ${}^{n-1} \Gamma_{i-1}|_{\delta'} \subset   \{x_1 = \delta'_0  h_{i-1}^2\}$
pull back to the same functions $h_{i, 1}, \ldots, h_{i, i-1}$ cutting out ${}^n \Gamma_i|_\delta$.

Finally, the   coorientations of ${}^{n} \Gamma_{i}|_{\delta}$,${}^{n-1} \Gamma_{i-1}|_{\delta'}$ are positive on respectively $ \partial_{x_0}$,  $ \partial_{x_1}$.  
Observe the $\partial_{x_1}$-component of $s_* \partial_{x_0}$ is in the direction of $ \partial_{x_1}$,
and hence $s$  gives a  cooriented identification.
\end{proof}

\subsubsection{Alternative presentation}\label{sect:alt sign}

For compatibility with inductive arguments, it is useful to introduce an alternative sign convention and alternative presentation of the local models.
%
%Given signs $\eps = (\eps_0, \ldots, \eps_n)$,  consider the signs 
%$$
%\xymatrix{
%\delta = \delta(\eps) =  (\delta_0, \ldots, \delta_n)
%&
%\delta_i =  \eps_{i-1}\eps_i
%}
%$$
%where by convention we set $\eps_{-1} = 1$.
%So for example, for $\eps = (\eps_0, \ldots, \eps_0)$, we have $\delta(\eps) = (\eps_0, 1, \ldots, 1)$.
%Note the inverse transformation $\eps_i = \delta_0 \cdots \delta_i$. 

%
%Set 
%$$
%\xymatrix{
%{}^n  \Gamma_i^\eps =    \ {}^n \Gamma_i|_{\delta(\eps)}  &  {}^n \Gamma^\eps =   {}^n  \Gamma|_{\delta(\eps)} 
%}
%$$
%
%

%
%
%
%Note we can restate Lemma \ref{lem: induction for fronts} in the following form. 
% Fix $\eps = (\eps_0, \eps_1, \ldots, \eps_n)$,  and set $\eps' = (\eps_1, \ldots, \eps_n)$
% The diffeomorphism $s$  of Lemma \ref{lem: induction for fronts} 
% gives a cooriented identification
% $$
% \xymatrix{
% s({}^n \Gamma^\eps_i \cap \{ \eps_0 x_0>0\} ) =  \R_{>0} \times {}^{n-1} \Gamma^{\eps'}_{i-1}  & 0 < i \leq n
% }
% $$
% 

Fix signs $\eps= (\eps_0, \ldots, \eps_n)$.
Consider the involution $\sigma_\eps: \R^n \to \R^n$ defined by $\sigma_\eps(x_1, \ldots, x_n) = (\eps_1 x_1, \ldots, \eps_n x_n)$. 

Define the  domain quadrant $ {}^n  R^\eps_i  \subset \R^n$ cut out by the inequalities
$$
\xymatrix{
\eps_0 \eps_1  h_{i, 0}\circ \sigma_\eps \leq 0, \ldots,  \eps_{i-1} \eps_i  h_{i, i-1}\circ \sigma_\eps \leq 0
}
$$

Define the cooriented hypersurface  ${}^n  \Gamma_i^\eps \subset \R^{n+1}$ to be
the restricted signed graph
$$
\xymatrix{
{}^n  \Gamma_i^\eps = \{x_0 = \eps_0 h_i^2 \circ \sigma_\eps\}|_{{}^n R^{\eps}_i} %\{ x_{0} = h_i(x_1, \ldots, x_i)^2 \, |\,  (x_1, \ldots, x_i) \in Q_i \}
}
$$
with the graphical coorientation.
Thus ${}^n  \Gamma_i^\eps$ is cut out by the equations
$$
\xymatrix{
x_0 =  \eps_0 h_i^2\circ \sigma_\eps &
\eps_0 \eps_1  h_{i, 0}\circ \sigma_\eps \leq 0, \ldots,  \eps_{i-1} \eps_i  h_{i, i-1}\circ \sigma_\eps \leq 0
}
$$
Consider as well the union
$$
\xymatrix{
 {}^n\Gamma^\eps = \bigcup_{i=0}^n   {}^n \Gamma^\eps_i
}
$$

\begin{remark}\label{rem: indep of last sign} 
A simple but important observation: ${}^n  \Gamma_i^\eps$  in fact only  depends on $\eps_0, \ldots, \eps_{i-1}$ and not $\eps_i$. This is because $h_{i, i-1} = x_i$ and so $\eps_{i-1} \eps_i  h_{i, i-1}\circ \sigma_\eps = \eps_{i-1}  x_i$. In particular, the union $ {}^n\Gamma^\eps$ is independent of $\eps_n$.
\end{remark}

We have the following adaption of Lemma~\ref{lem: induction for fronts}.

 \begin{lemma}\label{lem: induction for fronts 2}
 Fix $\eps = (\eps_0, \ldots, \eps_n)$,  and set $\eps' = (\eps_1, \ldots, \eps_n)$. The homeomorphism 
 $$
 \xymatrix{
 s:\eps_0\R_{\geq 0} \times \R^n \ar[r] & \eps_0\R_{\geq 0} \times \R^n 
 }
 $$
 $$
 \xymatrix{
  s (x_0, x_1, x_2,  \ldots, x_n) =  (x_0, x_1 +  \eps_0 \sqrt{\eps_0 x_0},x_2, \ldots, x_n)
 }
 $$
 gives a  cooriented identification
 $$
 \xymatrix{
 s({}^n \Gamma^\eps_i \cap \{ \eps_0 x_0 \geq 0\} ) =  \eps_0\R_{ \geq 0} \times {}^{n-1} \Gamma_{i-1}^{\eps'}  & 0 < i \leq n
 }
 $$
% and thus  a multi-cooriented identification
% $$
% \xymatrix{
% \tau({}^n P^\delta \cap \{ \delta_0 x_0>0\} ) = \delta_0 \R_{>0} \times {}^{n-1} P^{\delta'} 
% }
% $$
 \end{lemma}
 
\begin{proof}
Recall ${}^n \Gamma^\eps_i$ is defined by
$$
\xymatrix{
x_0 =  \eps_0 h_i^2\circ \sigma_\eps &
\eps_0 \eps_1  h_{i, 0}\circ \sigma_\eps \leq 0, \ldots,  \eps_{i-1} \eps_i  h_{i, i-1}\circ \sigma_\eps \leq 0
}
$$
in particular
$$
\xymatrix{
x_0 =  \eps_0 h_i^2 \circ \sigma_\eps  & \eps_0 \eps_1  h_{i, 0}\circ \sigma_\eps = \eps_0 \eps_1  h_{i}\circ \sigma_\eps \leq 0
}
$$
Note the functions $h_{i, 1}, \ldots, h_{i, i-1}$ are independent of the coordinates $x_0, x_1$.

When $  \eps_0 x_0 \geq 0$ and $\eps_0 \eps_1  h_{i}\circ \sigma_\eps \leq 0$, the equation $   x_0 = \eps_0 h_i^2\circ \sigma_\eps$  is equivalent to $\sqrt{   \eps_0 x_0} = -\eps_0\eps_1 h_i\circ\sigma_\eps$. Expanding this in terms of the definitions, we can rewrite this in the form
$$
\xymatrix{
x_1 + \eps_0 \sqrt{ \eps_0  x_0} = \eps_1 h_{i-1, 1}^2 \circ \sigma_{\eps'}
}
$$
Thus we see $s$ takes ${}^n \Gamma^\eps_i \cap \{ \eps_0  x_0 \geq 0\}$  into 
$\eps_0 \R_{ \geq 0} \times \{x_1 = \eps_1  h_{i-1, 1}^2 \circ \sigma_{\eps'}\}$. 

Moreover,  the additional functions $h_{i, 1}, \ldots, h_{i, i-1}$ cutting out $${}^{n-1} \Gamma_{i-1}^{\eps'} \subset  
\{x_1 = \eps_1  h_{i-1, 1}^2 \circ \sigma_{\eps'}\}$$
pull back to the same functions $h_{i, 1}, \ldots, h_{i, i-1}$ cutting out ${}^n \Gamma^\eps_i$.

Finally, the   coorientations of ${}^n \Gamma^\eps_i$,${}^{n-1} \Gamma_{i-1}^{\eps'} $ are positive on respectively $ \partial_{x_0}$,  $ \partial_{x_1}$.  
Observe the $\partial_{x_1}$-component of $s_* \partial_{x_0}$ is in the direction of $ \partial_{x_1}$,
and hence $s$  gives a  cooriented identification.
\end{proof}

Here is a useful corollary that ``explains" the geometric meaning of the signs $\eps$.

\begin{cor}
 Fix $\eps = (\eps_0, \ldots, \eps_n)$.
 
 For $i=0, \ldots, n-1$, we have $\eps_i = \pm 1$ if and only if ${}^n\Gamma_{i+1}$ is on the $\pm$-side of ${}^n \Gamma_i$ with respect to the graphical $dx_0$-coorientation.

Moreover, for $i=1, \ldots, n-1$, we have $\eps_i = \pm 1$ if and only if ${}^n\Gamma_{i+1} \cap {}^n \Gamma_0$ is on the $\pm$-side of ${}^n \Gamma_i\cap {}^n \Gamma_0$ with respect to the graphical $dx_1$-coorientation.

\end{cor}

\begin{proof}
For $i= 0$, the first assertion is immediate from the definitions ${}^n \Gamma_0 = \{x_0 = 0\}$ and ${}^n \Gamma_1 = \{x_0 = \eps_0 (\eps_1 x_1)^2 = \eps_0 x_1^2, \eps_0 \eps_1 (\eps_1 x_1)  = \eps_0 x_1 \leq 0 \}$.

For $i>0$, both assertions follow by induction from Lemma~\ref{lem: induction for fronts 2}.
\end{proof}

Fix signs $\eps= (\eps_0, \ldots, \eps_{n-1})$.
For $i = 0$, let ${}^{n} L^\eps_0 =   \R^{n} \subset T^*\R^{n}$ denote the zero-section.
For $i = 1, \ldots, n$,  introduce the positive conormal bundles
$$
\xymatrix{
{}^{n} L^\eps_{i} = T^+_{ {}^{n-1}\Gamma^\eps_{i-1}} \R^{n} \subset T^* \R^{n}
}
$$
determined by the graphical coorientation, 
 and consider their union
$$
\xymatrix{
{}^{n} L^\eps = \bigcup_{i = 0}^n {}^n L^\eps_i
}
$$

%Regard ${}^n\Gamma_i$ as cooriented by the positive $\partial_{x_0}$ direction
Fix signs $\eps= (\eps_0, \ldots, \eps_{n})$.
For $i = 0, \ldots, n$, introduce the  Legendrian  
$$
\xymatrix{
{}^{n} \Lambda^\eps_i \subset J^1 \R^n
}
$$
  projecting diffeomorphically to the front ${}^n\Gamma^\eps_i \subset \R^{n+1}$, 
and consider their union
$$
\xymatrix{
{}^n \Lambda^\eps = \bigcup_{i =0}^n {}^{n} \Lambda^\eps_i
}
$$

We have the following compatibility of the above Lagrangians and Legendrians analogous to Lemma~\ref{lem: tilt}.

%Let $p = (p_1, \ldots, p_n)$ be dual coordinates to $x= (x_1, \ldots, x_n)$, and let $dx_0-\sum\limits_1^np_idx_i$ be  the contact form on $J^1\R^n$.
 
\begin{lemma}\label{lem: tilt signs}
Fix signs $\eps = (\eps_0, \ldots, \eps_n)$, and set $\eps' = (\eps_1, \ldots, \eps_n)$. The contactomorphism
$$
\xymatrix{
S_{\eps_0}: J^1 \R^n \ar[r] &  J^1 \R^n
}
$$
$$
\xymatrix{
S_{\eps_0}(x_0, x, p) = (x_0 - \eps_0 p_1^2/4, x_1 + \eps_0 p_1/2, x_2, \ldots, x_n,   p_1, \ldots,  p_n)
}
$$
takes the  Legendrian  ${}^{n} \Lambda^\eps_i $  isomorphically to
the Legendrian
$\{0\} \times {}^n L^{\eps'}_i$, and
thus the union ${}^{n} \Lambda^\eps$  isomorphically to the union
$\{0\} \times {}^n L^{\eps'}$.

\end{lemma}

\begin{proof} 
The proof is the same as that of Lemma~\ref{lem: tilt} with the following observations.
Consider the additional equations
$$
\xymatrix{
\eps_0 \eps_1 \delta_1 h_{i, 0}\circ \sigma_\delta \leq 0, \ldots, \eps_{i-1}  \eps_i h_{i, i-1}\circ \sigma_\eps \leq 0
}
$$
First, over $\eps_0 \eps_1 h_{i, 0} \circ \sigma_\eps\leq 0$, when $p_1 = -2 \eps_0\eps_1  h_{i, 0} \circ \sigma_\eps$, we then have 
$p_1 = -2\eps_0 \eps_1 h_{i, 0}  \circ \sigma_\eps \geq 0$, so we obtain the positive conormal direction.  
Second, the remaining functions $h_{i, 1}, \ldots, h_{i, i-1}$ 
are independent of $x_0, x_1$.
Thus $S_{\eps_0}$ indeed takes ${}^{n}\Lambda^\eps_i$ to $\{0\} \times {}^{n}L^\eps_i$.
\end{proof}

\begin{remark}
By the lemma, we see  the  Legendrian  ${}^{n} \Lambda^\eps_i \subset J^1 \R^n$ is independent of the initial sign $\eps_0$ so only depends on $\eps' = (\eps_1, \ldots, \eps_n)$.
\end{remark}

It is also useful to record the following relationship of $ {}^n\Gamma^\eps $ with the extended model 
$ {}^n\Gamma$. 

%Write $\eps\cdot {}^n \Gamma_i = \sigma_\eps(\eps_0 \cdot {}^n \Gamma_i) \subset \R^n$
%where $\pm 1\cdot  {}^n\Gamma_i \subset \R^{n+1}$  denotes the usual graph (+1) or   negated graph (-1), and set $\eps\cdot {}^n \Gamma = \cup_i \eps\cdot {}^n \Gamma_i \subset \R^{n+1}$.
%Similarly, write $\eps\cdot {}^n \Lambda_i \subset J^1\R^n$ for 
% the Legendrian projecting to $\eps\cdot {}^n \Gamma_i \subset \R^{n+1}$, and
% set  $\eps \cdot {}^n\Lambda = \cup_i \eps\cdot {}^n \Lambda_i$ 
% 
%

%Write $\pm 1\cdot  {}^n\Gamma_i \subset \R^n$  for the usual graphs $(+1)$ or  the negated graphs $(-1)$,  with their  graphical coorientation. Set $\pm 1\cdot  {}^n\Gamma = \cup_i \pm 1\cdot  {}^n\Gamma_i$ 
%to be their union. Similarly, write $\pm 1\cdot {}^n\Lambda_i \subset J^1 \R^n$ for the Legendrian projecting to $\pm 1\cdot  {}^n\Gamma_i \subset \R^n$. Set $\pm 1\cdot {}^n\Lambda = \cup_i {}^n \Lambda_i$ to be their union.

\begin{lemma} \label{lem: ext comp}
Fix signs $\eps= (\eps_0, \ldots, \eps_n)$.

Given a contactomorphism $J^1 \R^n\to J^1\R^n$ restricting to a closed embedding ${}^n \Lambda^{\eps} \subset \eps_0\cdot {}^n \Lambda$ with
${}^n \Lambda_i^{\eps} \subset \eps_0\cdot {}^n \Lambda_i$, for all $i$, consider the front $\Upsilon = \pi( {}^n \Lambda^{\eps}) \subset \eps_0 \cdot {}^n \Gamma$. 

Then either the involution $\sigma_\eps$ or its composition with $x_n \mapsto \pm x_n$ takes $\Upsilon$ to ${}^n \Gamma^\eps$.

\end{lemma}

\begin{proof}
Note we have ${}^n \Lambda^\eps_0 = \eps_0 \cdot {}^n \Lambda_0 = {}^n \Lambda_0$. Consider the intersection $\Upsilon' = \pi(({}^n \Lambda^\eps \setminus  {}^n \Lambda_0) \cap
{}^n \Lambda_0)$
as a front inside of $\pi({}^n \Lambda_0) = {}^n \Gamma_0 = \{ x_0 = 0\}$. By induction, 
either the involution $\sigma_\eps$ or its composition with $x_n \mapsto \pm x_n$ takes $\Upsilon'$ to ${}^{n-1} \Gamma^{\eps'}$ where $\eps' = (\eps_1, \ldots, \eps_n)$. So we may assume  $\Upsilon' = {}^{n-1}  \Gamma^{\eps'}$. Now observe ${}^n\Gamma^\eps$ is the unique way to extend ${}^{n-1} \Gamma^{\eps'}$ within $\sigma_\eps(\eps_0 \cdot {}^n \Gamma)$ compatible with coorientations. 
\end{proof}

We also have the following observation about signs. See Section \ref{ss:arb lags and legs} for notation.

\begin{lemma} 

Let $\nu_0$ be the vertical polarization of $T^*\R^n \to \R^n$.

Then we have $\eps(\nu_0,{}^n L^\eps_1,  {}^n L^\eps_2) = \eps_0$.
\end{lemma}

\begin{proof}
Recall ${}^n L^\eps_1$ is the positive conormal to the graph ${}^{n-1} \Gamma^\eps_0 = \{x_0 = 0\}$,
and ${}^n L^\eps_1$ is the positive conormal to the graph ${}^{n-1} \Gamma^\eps_1 =\{ x_0 = \epsilon_0 x_1^2\}$.
Since $ \eps_0 x_1^2$ is an $\eps_0$-definite quadratic form in $x_1$, the assertion follows.
\end{proof}

%%%%%%%%%%%%%%%%%%%%%%%%%%%%%%%%%
\subsection{Arboreal models}

We now present the local models for arboreal singularities.

\subsubsection{Signed rooted trees}

\begin{definition}  We will use the following terminology throughout:
\begin{enumerate}
\item

A {\em   tree} $T$ is a nonempty, finite, connected acyclic graph.

\item A {\em rooted tree} $\sT = (T, \rho)$ is a pair of a tree $T$ and a distinguished vertex $\rho$ called the {\em root}.

\item
A {\em  signed rooted tree} $\cT = (T, \rho, \eps)$ is a rooted tree $(T, \rho)$
and a decoration $\eps$ of a sign $\pm 1 $ on each edge of $T$ not adjacent to the root $\rho$.

\end{enumerate}
\end{definition}

    \begin{figure}[h]
\includegraphics[scale=0.6]{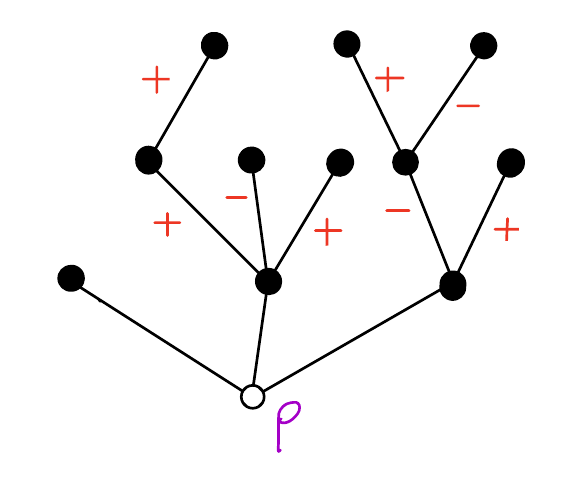}
\caption{A signed rooted tree.}
\label{fig:signedrootedtree}
\end{figure}

Given a signed rooted tree $\cT = (T, \rho, \eps)$, 
we write $v(T)$ for the  set of vertices, $e(T)$ for the   set of edges, 
and $n(\cT) = v(T) \setminus \rho$ for the set of non-root vertices.
% and $|n(\cT)|$ for the number of elements of $n(\cT)$.
We regard $v(T)$ as  a poset with  unique minimum $\rho$, and in general
 $\alpha \leq  \beta\in v(T)$ when 
 the shortest path connecting $\beta$ and $\rho$ contains $\alpha$. 
We call a non-root vertex  $\beta$ a  leaf if exactly one edge of $T$ is adjacent to $\beta$, and write
$\ell(\cT) \subset v(T)$ for the set of leaf vertices.

%
%\begin{remark}
%To each $\alpha \in n(\cT)$, there is a unique vertex $ \hat\alpha \in v(T)$ and edge $[\alpha, \hat \alpha] \in e(T)$ such that $ \hat\alpha< \alpha$. Thus one could equivalently interpret $\eps$ as  the decoration of a sign $\pm 1 $ on each edge not adjacent to a leaf vertex.
%\end{remark}
%
 \begin{remark}
 Throughout what follows, for a finite set $S$, we write $\R^S$ for the Euclidean space of $S$-tuples of real numbers.
 One may always fix a bijection $S \simeq \{1,2, \ldots, n\}$, for some $n\geq 0$, and hence an isomorphism $\R^S \simeq \R^n$, but it will be convenient to avoid  choosing such identifications when awkward.
 We will most often consider $S = n(\cT)$  the non-root vertices  for some rooted tree $\cT = (T, \rho)$. Here if one prefers to 
 fix a bijection $b:n(\cT) \risom \{1,2, \ldots, |n(\cT)|\}$, we recommend choosing $b$ to be order-preserving: if $\alpha\leq \beta$, then one should ensure $b(\alpha) \leq b(\beta)$. This will allow for a clear translation of our constructions.
 \end{remark}
% 
%
%We will work with the cotangent bundle $T^*\R^S$ with the standard  Liouville form $\lambda = pdq$,
%symplectic form $\omega = d\lambda$, and Liouville vector field $Z = \omega^{-1}(\lambda)$. 
%%We write $\R^S \subset T^*\R^S$ for the zero-section, or $T^*_{\R^S} \R^S \subset T^*\R^S$ when some confusion %might otherwise result. 
%We  follow usual conventions and say a smooth submanifold $Y\subset T^*\R^S$ is $\lambda$-isotropic if $\lambda|_Y = 0$, $\omega$-isotropic if $\omega|_Y = 0$, $Z$-invariant if $Z$ is tangent to $Y$, and conic if $Y$ is invariant under the flow of $Z$. Recall that $Y$ is $\lambda$-isotropic if and only if $\omega$-isotropic and $Z$-invariant. We similarly say a submanifold with corners $Y\subset T^*\R^S$ is $\lambda$-isotropic, $\omega$-isotropic, $Z$-invariant, or conic when its interior $\openY \subset T^*\R^S$ and boundary corners are respectively so.
% 
%
%Similarly, 
%we will work with the contact  space $J^1 \R^S = \R \times T^*\R^S$ with  contact form $\hat \lambda = dz + pdq$, 
% and front projection $\pi:\R \times T^*\R^{S}\to \R^{S}, \pi(z, q, p) = q$.
%We say a smooth submanifold $Y\subset \R \times T^*\R^{S}$ is $\hat \lambda$-isotropic if $\hat \lambda|_Y = 0$.
%Recall that a submanifold $Y \subset T^*\R^{S}$ is $\lambda$-isotropic  if and only if $\{0\} \times Y\subset \R\times T^*\R^{S}$ is $\hat \lambda$-isotropic. 
%We similarly say a submanifold with corners $Y\subset \R \times T^*\R^{S}$ is $\hat \lambda$-isotropic when its interior $\openY \subset \R \times T^*\R^{S}$ and boundary corners are so.
%
%
\begin{definition} A signed rooted tree $\cT=(T,\rho,\eps)$ is called {\em positive} if  the decoration $\eps$ consists of signs $+1$.
\end{definition}

We will associate  to any signed rooted tree $\cT = (T, \rho, \eps)$, a multi-cooriented hypersurface, conic Lagrangian,
and Legendrian 
$$
\xymatrix{
H_\cT \subset \R^{n(\cT)} & L_\cT \subset T^*\R^{n(\cT)} & \Lambda_\cT \subset J^1\R^{n(\cT)} 
}
$$  
where as usual we write $n(\cT) = v(T) \setminus \rho$ for the set of non-root vertices.

By definition, the latter two will be determined by the first as follows:
%Given a rooted tree $\cT = (T, \rho)$, denote by
%$\cT^* = (T^*, \rho^*)$ the rooted tree where we adjoin a new root. Thus  we have one additional vertex $v(T^*) = v(T) \cup \{\rho^*\}$, one additional edge $e(T^*) = e(T) \cup \{[\rho^*, \rho]\}$, and the new root $\rho^*$.

\begin{enumerate}

\item $L_\cT$ is the union of the zero-section $\R^{n(\cT)}$ and the positive conormal to $H_\cT$.

\item $\Lambda_{\cT}$  is the Legendrian lift of $L_\cT$ with zero primitive.

\end{enumerate}

\subsubsection{Type $A$ trees}
Let us first consider the distinguished case of $A_{n+1}$-trees with extremal root.

\begin{definition}

For $n\geq 0$, 
a {\em linear signed $A_{n+1}$-rooted tree} is a signed rooted tree $\cA_{n+1} =(A_{n+1}, \rho, a)$ with vertices $v(A_{n+1}) =\{0, 1, \ldots, n\}$, edges $v(A_{n+1}) = \{ [i, i+1] \,|\, i = 0, \ldots, n-1\}$, and root $\rho = 0$.
\end{definition}

By definition, the sign $a$ is a length $n-1$ list of signs $(a_{[1, 2]}, \ldots, a_{[n-1, n]})$.
Let us set $\eps = (\eps_0, \ldots, \eps_{n-1}) =  (a_{[1, 2]}, \ldots, a_{[n-1, n]}, 1)$ to be the length $n$ list of signs where we pad $a$ by adding a single $1$ at the end.

%
%\begin{remark}
%Note $\cA_{n+1}^* = \cA_{n+2}$.
%\end{remark}

\begin{definition}\label{def: type A} The models for $A_n$-type arboreal singularities are given as follows:

\begin{enumerate}

\item The arboreal $\cA_{1}$-front is the empty set $H_{\cA_1}  = \emptyset$ inside the point~$\R^0$. 

For $n\geq 1$, the arboreal $\cA_{n+1}$-front is the cooriented hypersurface 
$$
\xymatrix{
H_{\cA_{n+1}}  = {}^{n-1} \Gamma^\eps  \subset \R^n
}
$$
introduced in Section~\ref{sect:alt sign}.

\item %The arboreal $\cA_{1}$-Lagrangian is the point $\Lambda_{\cA_1} = \Lambda_1 = \R^0$ inside of the point $T^*\R^0$. 

For $n\geq 0$, the arboreal $\cA_{n+1}$-Lagrangian is the union of the zero-section and positive conormal
$$
\xymatrix{
L_{\cA_{n+1}} = \R^n \cup  T^+_{  \R^n} H_{\cA_{n+1}} \subset T^*\R^{n}
}
$$

\item %The arboreal $\cA_{1}$-Legendrian is the point $L_{\cA_1} = L_1 = \R^0$ inside of the line $J^1 \R^0 = \R$.

For $n\geq 0$, the arboreal $\cA_{n+1}$-Legendrian  is the lift 
$$
\xymatrix{
\Lambda_{\cA_{n+1}} = \{0\} \times L_{\cA_{n+1}}  \subset J^1 \R^n %  \subset T^\infty(\R^{n+1})
}
$$
 \end{enumerate}

\end{definition}

    \begin{figure}[h]
\includegraphics[scale=0.6]{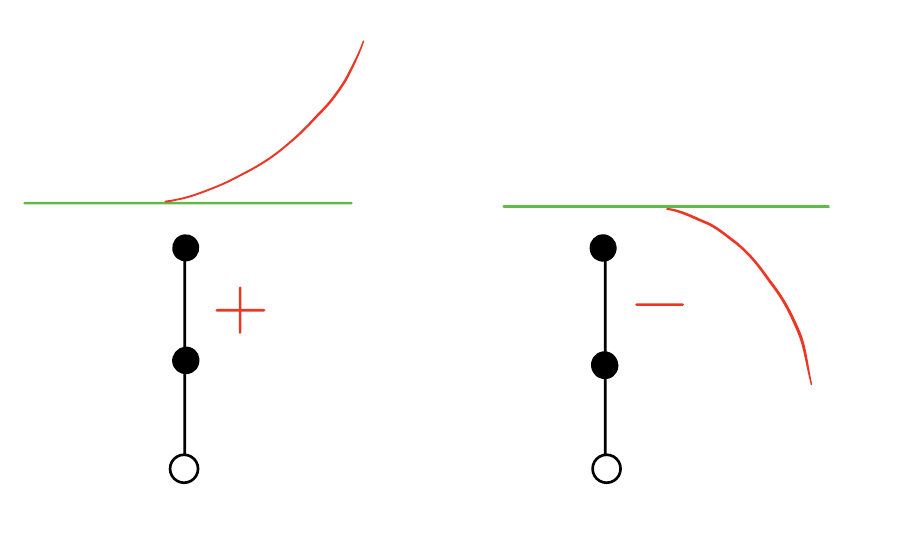}
\caption{The two $A_3$ fronts with positive and negative sign.}
\label{fig:quadrants}
\end{figure}

    \begin{figure}[h]
\includegraphics[scale=0.5]{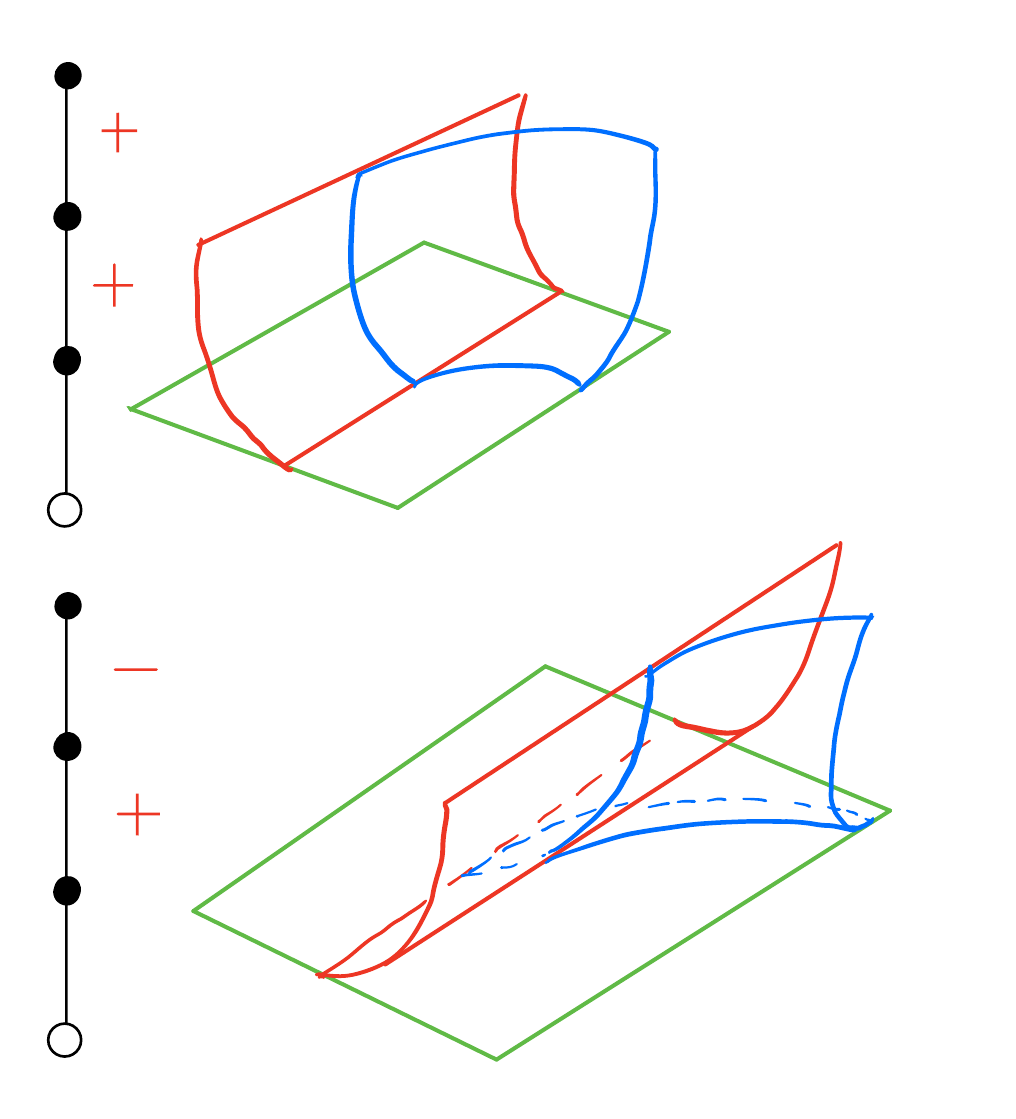}
\caption{Two $A_4$ fronts with different choices of signs. The other two fronts can be obtained from these two by reflections.}
\label{fig:posvneg}
\end{figure}

\begin{remark}
Following Remark~\ref{rem: indep of last sign}, the arbitrary choice of the last sign $\eps_{n-1} = 1$ does not affect the  arboreal $\cA_{n+1}$-models. 
\end{remark}

%\begin{remark}\label{rem: indexing}
Recall the linear signed $A_{n+1}$-rooted tree $\cA_{n+1} = (A_{n+1}, \rho, a)$ has vertices $v(A_{n+1}) =\{0, 1, \ldots, n\}$ with root $\rho = 0$, and 
so the non-root vertices form the set  $n(\cA_{n+1}) =\{1, \ldots, n\}$.
In the above definition, we should 
more invariantly  view the ambient Euclidean space $ \R^{n}$ in the form $\R^{n(\cA_{n+1})}$ where the ordering of the coordinates matches that of $n(\cA_{n+1})$. 

With this viewpoint, we rename the smooth pieces  of the $\cA_{n+1}$-front, indexing them by non-root vertices
$$
\xymatrix{
H_{i} = {}^{n-1} P^\eps_{i-1} \subset H_{\cA_{n+1}} & i \in n(\cA_{n+1}) =\{1, \ldots, n\}
}
$$
Likewise, we rename the smooth pieces of the of the $\cA_{n+1}$-Lagrangian, indexing them by   vertices
$$
\xymatrix{
L_{0} = \R^n \subset L_{\cA_{n+1}} 
}
$$
$$
\xymatrix{
L_{i} = T^+_{  \R^n} H_{i}  \subset L_{\cA_{n+1}} & i \in n(\cA_{n+1}) =\{1, \ldots, n\}
}
$$ 
and similarly, we rename
the smooth pieces of the of the $\cA_{n+1}$-Legendrian, indexing them by  vertices
$$
\xymatrix{
\Lambda_{i}  = \{0\} \times  L_{\cA_{n+1}, i} \subset \Lambda_{\cA_{n+1}} &  i  \in v(\cA_{n+1}) =\{0, 1, \ldots, n\} 
}
$$ 
%\end{remark}

\begin{lemma}\label{lem: leaf pruning}
For $n \geq 1$, and $n \in v(A_{n+1}) = \{0, 1, \ldots, n\}$ the unique leaf vertex, and $\mathring H_{n}  \subset H_{\cA_{n+1}} $ the interior of the corresponding smooth piece, we have
$$
\xymatrix{
H_{\cA_{n+1}} \setminus \mathring H_{n} = H_{\cA_n} \times \R
}
$$
inside of $\R^{n(\cA_{n+1}) } = \R^{n(\cA_n)} \times \R$.
\end{lemma}

\begin{proof}
Recall the other smooth pieces $H_{i} = {}^{n-1} P^\eps_{i-1}$, for $i = 1, \ldots, n-1$, are independent of the last coordinate $x_n$.
\end{proof}

\subsubsection{General trees}
Now we consider a general signed rooted tree $\cT = (T, \rho, \eps)$. 
%Recall we regard the vertices $v(T)$ as  a poset with  unique minimum $\rho$, and in general
% $\alpha \leq  \beta\in v(T)$ when 
% the shortest path connecting $\beta$ and $\rho$ contains $\alpha$. 
%Recall we call a non-root vertex  $\beta \in n(\cT)$ a leaf if exactly one edge of $T$ is adjacent to $\beta$, and
%write $\ell(\cT) \subset n(\cT)$ for the set of leaf vertices. 

To each leaf $\beta \in \ell(\cT)$, we associate the linear signed  $A_{n+1}$-rooted tree $\cA_\beta = (A_\beta, \rho, a)$ 
 where $A_\beta$ is the full subtree of $T$ on the  vertices $v(A_\beta) = \{ \alpha \leq \beta \in v(T)\}$, and $a$ is the restricted sign decoration.

Consider the Euclidean space $\R^{n(\cT)}$. For each $\beta \in \ell(\cT)$, the inclusion $n(\cA_\beta) \subset n(\cT)$ induces a natural projection
$$
\xymatrix{
\pi_\beta: \R^{n(\cT)} \ar[r] & \R^{n(\cA_\beta)}
}
$$

\begin{definition}\label{def: arb model}
 Let $\cT = (T, \rho, \eps)$ be a signed rooted tree.

\begin{enumerate}

\item The arboreal model $\cT$-front is the multi-cooriented hypersurface given by the union
$$
\xymatrix{
H_{\cT} =  \bigcup_{\beta \in \ell(\cT)}  \pi_\beta^{-1}(H_{\cA_\beta})  \subset \R^{n(\cT)}
}
$$
where $H_{\cA_\beta} \subset \R^{n(\cA_\beta)}$ is the arboreal $\cA_{\beta}$-front.

\item The arboreal model $\cT$-Lagrangian is the union of the zero-section and positive conormal
$$
\xymatrix{
L_{\cT} = \R^{n(\cT)} \cup  T^+_{\R^{n(\cT)}} H_{\cT} \subset T^*\R^{n(\cT)}
}
$$

\item The arboreal  model $\cT$-Legendrian  is the the lift
$$
\xymatrix{
\Lambda_{\cT} = \{0\} \times L_\cT \subset  J^1 \R^{n(\cT)}
}
$$
\end{enumerate}

 Arboreal models $H_{\cT},L_{\cT}$ and $\Lambda_{\cT}$ corresponding to positive $\cT$ are called positive.
 
\end{definition}

      \begin{figure}[h]
\includegraphics[scale=0.5]{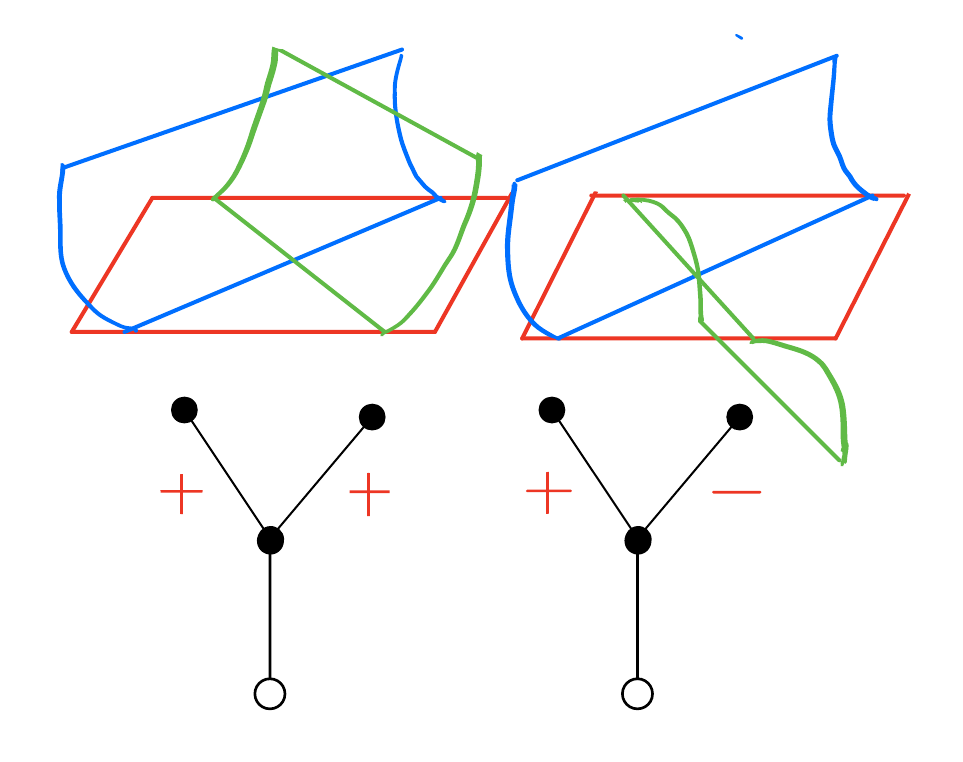}
\caption{Two non $A_n$-type fronts with different choices of signs.}
\label{fig:signsagain}
\end{figure}

\begin{remark}
When $\cT = \cA_{n+1}$, the above definition recovers Definition~\ref{def: type A} verbatim.
\end{remark}

Transporting  from the case of $\cA_{n+1}$,  we may naturally index the smooth pieces  of the $\cT$-front by non-root vertices
$$
\xymatrix{
H_{\alpha}  = \pi_\beta^{-1}(H_{\cA_{\beta}, \alpha}) \subset H_{\cT} & \alpha \in n(\cT)
}
$$
where $\beta \in \ell(\cT)$ is any leaf with $\alpha\leq \beta$, and $H_{\cA_{\beta}, \alpha} \subset  H_{\cA_\beta}$ is the corresponding smooth piece.
Likewise, we may index the smooth pieces of the $\cT$-Lagrangian  by vertices
$$
\xymatrix{
L_\rho = \R^{n(\cT)} \subset L_{\cT} 
}
$$
$$
\xymatrix{
L_{\alpha} =   T^+_{\R^{n(\cT)}} H_{\alpha}  \subset L_{\cT} & \alpha \in n(\cT) 
}
$$ 
and 
the smooth pieces of the $\cT$-Legendrian  by vertices
$$
\xymatrix{
\Lambda_{\alpha}   = \{0\} \times L_\alpha  \subset \Lambda_{\cT} &  \alpha  \in v(\cT) 
}
$$

Let us record a basic compatibility of the above Lagrangians and Legendrians.

Fix a signed rooted tree $\cT = (T, \rho, \eps)$. %Set $n = |n(\cT)|$.  
Let us first consider the situation when there is a single vertex $\rho' \in \cT$ adjacent to $\rho$.
Let $\cT' = \cT \setminus \rho$ be the signed rooted tree with root $\rho'$ and restricted signs. 

Let $\alpha_1, \ldots, \alpha_k \in \cT'$ be the vertices adjacent to $\rho'$, and $\eps_1, \ldots, \eps_k$ the signs of $\cT$ assigned to the respective edges from $\rho'$ to  $\alpha_1, \ldots, \alpha_k$.

Let $L^\infty_\cT \subset S^* \R^{n(\cT)}$ be the ideal Legendrian boundary of $L_\cT \subset T^* \R^{n(\cT)}$.
Note that $L^\infty_\cT$ lies in the open subspace $J^1 \R^{n(\cT')} \simeq \{ p_{\rho'} = 1\} \subset S^* \R^{n(\cT)}$.

\begin{lemma}\label{lem: tilt arb}
The  contactomorphism
$$
\xymatrix{
S:  J^1\R^{n(\cT')} \ar[r] &   J^1\R^{n(\cT')}
}
$$
$$
\xymatrix{
S(x_{\rho'}, x, p) = (x_{\rho'} - \sum_{i = 1}^k \eps_i p_{\alpha_i}^2/4,  \hat x, p)
}
$$ 
$$
\xymatrix{
\hat x_{\alpha_i} = x_{\alpha_i} + \eps_i p_1/2, \mbox{ for } i = 1, \ldots, k,  &  \hat x_\beta = x_\beta \mbox{ else }
}
$$
takes the  Legendrian  $L^\infty_{\cT}$  isomorphically to
the Legendrian
$\{0\} \times L_{\cT'} $. 

Thus $L^\infty_{\cT}$ itself is a model arboreal Legendrian of type $\cT' = \cT \setminus \rho$.

 \end{lemma}

\begin{proof}
For each leaf vertex of $\cT$, we have a linear signed type $\cA$ subtree of $\cT$ given by the vertices running from $\rho$ to the leaf. By Definition~\ref{def: arb model}, 
$L_{\cT}$  is the union of the corresponding linear signed type $\cA$ subcomplexes 
$L_{\cA}$. Each such subcomplex is independent of the coordinate $x_\beta$ indexed by vertices $\beta$ not in the subtree, hence lies in the zero locus of the dual coordinate $p_\beta$. Thus transport of each $L^\infty_{\cA}$ under the contactomorphism of the lemma reduces to that of  Lemma~\ref{lem: tilt signs}.
\end{proof}

More generally, suppose $\rho_1,\dots, \rho_\ell$ are the vertices adjacent to $\rho$. 
Observe that $\cT \setminus \rho$ is a disjoint union of signed rooted subtrees  $\cT_j \subset \cT \setminus \rho$,
 for $j=1,\dots, \ell$, 
%, for $j=1,\dots, k$,  
with $\rho_j$ as root and restricted signs. 
Let $\cT^+_j = \cT_j \cup \rho \subset \cT$ be the 
signed rooted subtree with $\rho$ readjoined as root and with restricted signs. 
%Set $n_j = |n(\cT_j)|$. 
Set $c_j = n(\cT) \setminus n(\cT_j)$.

Let $L^\infty_\cT \subset S^* \R^{n(\cT)}$ be the ideal Legendrian boundary of $L_\cT \subset T^* \R^{n(\cT)}$.
We  similarly have $L^\infty_{\cT^+_j}  \subset S^* \R^{n(\cT_j^+)}$
 the ideal Legendrian boundary of  $L_{\cT_j^+} \subset T^* \R^{n(\cT_j^+)}$.

Since $\rho_j$ is the unique vertex adjacent to $\rho$ within $\cT_j^+$, observe that  $L_{\cT_j^+} $ is connected and in fact lies in $$J^1\R^{n(\cT_j)} =  \{p_{\rho_j} = 1\}  \subset S^* \R^{n(\cT_j^+)}.$$ Moreover, 
observe that $L^\infty_\cT$ is the disjoint union of the connected components 
$$
\Lambda_j = L^\infty_{\cT^+_j} \times \R^{c_j} \subset J^1\R^{n(\cT_j)} \times T^*\R^{c_j}  = \{p_{\rho_j} = 1\} \subset S^* \R^{n(\cT)}
$$ 
By Lemma~\ref{lem: tilt arb}, $L^\infty_{\cT^+_j} \subset J^1\R^{n(\cT_j)}$  is a model aboreal Legendrian of type $\cT_j$, so $\Lambda_j = L^\infty_{\cT^+_j} \times \R^{c_j} \subset J^1\R^{n(\cT_j)} \times T^*\R^{c_j}$ is a stabilized model arboreal Legendrian of type $\cT_j$. This proves:

\begin{lemma} \label{lm: can models are arb}
Fix a signed rooted tree $\cT = (T, \rho, \eps)$. %Set $n = |n(\cT)|$.

Let $\rho_1,\dots, \rho_k$ be the vertices adjacent to $\rho$. 
Let $\cT_j \subset \cT \setminus \rho$ be the signed rooted subtree with $\rho_j$ as root and restricted signs,
and 
 $\cT^+_j = \cT_j \cup \rho \subset \cT$  the 
signed rooted subtree with $\rho$ readjoined as root and with restricted signs. 
Set $c_j = n(\cT) \setminus n(\cT_j)$.

Then the ideal Legendrian boundary  $L^\infty_\cT \subset S^* \R^{n(\cT)}$  of the model arboreal Lagrangian
 $L_\cT \subset T^* \R^{n(\cT)}$ of type $\cT$ is the disjoint union of the Legendrians
 $$\Lambda_j  = L^\infty_{\cT^+_j}  \times \R^{c_j} \subset S^* \R^{n(\cT)},$$
 which are stabilized model  arboreal Legendrians  of type $\cT_j$.
 \end{lemma}

By Lemma~\ref{lem: leaf pruning}, 
we also  have the following.

\begin{cor}\label{cor:int-arb}
For  $\beta \in \ell(\cT)$ a leaf vertex,  and $\mathring H_{\beta}  \subset H_{\cT} $ the interior of the corresponding smooth piece, we have
 $$
\xymatrix{
H_{\cT} \setminus \mathring H_\beta = H_{\cT \setminus \beta} \times \R^\beta
}
$$
inside of $\R^{n(\cT)} = \R^{n(\cT \setminus \beta)} \times \R^\beta$.
\end{cor}

 \subsubsection{Extended arboreal models}
 It will be useful for us also define {\em extended arboreal models} associated with rooted, but not signed trees $\sT=(T,\rho)$.
 
For the unsigned   rooted  tree $\sA_{n+1}=(A_{n+1},\rho)$ we define 
$$H_{\sA_{n+1}}:={}^{n-1}\Gamma\subset\R^n,$$ 
$$L_{\sA_{n+1}}:=\R^n\cup T^*_{\R^n}H_{\sA_{n+1}}\subset T^*\R^n,$$
$$\Lambda_{\sA_{n+1}}:=0\times L_{\sA_{n+1}}\subset J^1\R^n.$$

Similarly, for a general rooted tree $\sT=(T,\rho)$ we define

$$
\xymatrix{
H_{\sT} =  \bigcup_{\beta \in \ell(\sT)}  \pi_\beta^{-1}(H_{\sA_\beta})  \subset \R^{n(\cT)}
}
$$
where $H_{\sA_\beta} \subset \R^{n(\sA_\beta)}$ is the arboreal $\sA_{\beta}$-front.
Furthermore, we define
 $$
\xymatrix{
L_{\sT} = \R^{n(\sT)} \cup  T^+_{\R^{n(\sT)}} H_{\cT} \subset T^*\R^{n(\sT)}
}
$$
and
$$
\xymatrix{
\Lambda_{\sT} = \{0\} \times \Lambda_\sT \subset  J^1 \R^{n(\sT)}
}
$$
 
 Clearly, for any signed version $\cT$ of the tree $\sT$ we have
 $H_\cT\subset H_\sT, L_\cT\subset L_\sT, \Lambda_\cT\subset \Lambda_\sT.$

\begin{lemma} \label{lem: ext comp tree}

Given  a closed embedding $\Lambda^\infty_\cT  \subset \Lambda^\infty_\sT$ with
$\Lambda^\infty_{\cT,  \alpha} \subset \Lambda^\infty_{\sT, \alpha}$, for all $\alpha$,
 the front $\pi(\Lambda^\infty_\cT) \subset H_\sT$ is an embedding of $H_{\cT}$.
 
\end{lemma}

\begin{proof}
For each leaf vertex of $\cT$, we have a linear signed type $\cA$ subtree of $\cT$ given by the vertices running from $\rho$ to the leaf. By construction, 
$\Lambda^\infty_{\cT}$ and $\Lambda^\infty_{\sT}$  are the union of the corresponding type $\cA$ subcomplexes 
$L^\infty_{\cA}$ and $L^\infty_{\sA}$. Each such subcomplex is independent of the coordinates $x_\beta$ indexed by vertices $\beta$ not in the subtree. Now Lemma~\ref{lem: ext comp} confirms $\pi(L^\infty_{\cA})$ is 
 the standard  embedding of $H_{\cA}$ after a change of coordinates
 $x_\alpha$ indexed by vertices $\alpha$ in the subtree. Moreover, the change of coordinates agrees for  
  $x_\alpha$ indexed by vertices $\alpha$ in the intersection of such subtrees.
  By definition, $H_{\cT}$ is the union of the $H_{\cA}$.
\end{proof}

%%%%%%%%%%%%

\section{The stability theorem}\label{sec:stab}

In this section we  define  arboreal Lagrangian and Legendrian subsets and prove    their stability under   symplectic  reduction  and   Liouville cone operations.

\subsection{Arboreal Lagrangians and Legendrians}\label{ss:arb lags and legs}
\begin{definition} Arboreal Lagrangians and Legendrians are defined as follows:
\begin{itemize}
\item[(a)] A closed subset $L\subset X$ of a $2m$-dimensional symplectic manifold $(X,\om)$ is called an {\em arboreal Lagrangian} if  the germ of $(X, L)$ at any point $\lambda\in L$  is symplectomorphic to the germ of the pair $(T^*\R^{n}\times T^*\R^{m-n},L_{\cT}\times \R^{m-n})$  at the origin,
for a signed rooted tree $\cT$  with $n:=n(\cT)\leq m$.

\item[(b)] A closed
subset $\Lambda\subset Y$ of a $(2m+1)$-dimensional contact manifold $(Y,\xi)$ is called am {\em arboreal Legendrian} if  the  germ of  $(Y,\Lambda)$ at any point $\lambda\in\Lambda$   is contactomorphic to the germ of  $(J^1(\R^{n}\times\R^{m-n})=J^1\R^{n }\times  T^*\R^{m-n},\Lambda_{\cT}\times\R^{m-n})$ at the origin,  
for a signed rooted tree $\cT$  with $n:=n(\cT)\leq m$.

\item[(c)] A closed  subset $H\subset M$ of an  $(m+1)$-dimensional manifold $M$ is called an {\em arboreal front}
 if the germ of $(M, H)$ at any point $m \in M$ is diffeomorphic to the germ of $(\R^{n+1} \times \R^{m-n}, H_\cT \times \R^{m-n})$
 at the origin, for a signed rooted tree $\cT$  with $n:=n(\cT)\leq m$.
\end{itemize}

The pair $(\cT,m)$ is called the {\em arboreal type} of the germ of  $L$, $\Lambda$, or $H$ at the given point.
We say $L$, $\Lambda$, or $H$ is {\em positive} if it is locally modeled  on positive arboreal models at all points.
\end{definition}

\begin{remark}
Later we will also allow arboreal Lagrangians to have boundary and even corners, but throughout the present discussion we restrict to the above definition for simplicity.
\end{remark}

Given an arboreal Lagrangian  we call
 $\sup_{\lambda\in L}  \{n(\cT(\lambda))\}$ the {\em maximal order} of $L$, where $\cT(\lambda)$ is a the signed  rooted tree describing the  germ of $L$ at the point $\lambda$. Similarly,  we define  the maximal order of arboreal Legendrians and fronts.
 
Every arboreal Lagrangian or Legendrian  is naturally stratified by isotropic strata
indexed by the corresponding tree type.
A Lagrangian distribution $\eta$  in $X$ is called transverse to an arboreal Lagrangian $L$ if it is transverse  to all top-dimensional strata of $L$. Similarly a Legendrian distribution $\eta\subset\xi$   in  a contact $(Y,\xi)$ is called transverse to an arboreal Legendrian $\Lambda$ if it has trivial intersection with tangent planes to all  top-dimensional strata of $\Lambda$. 

\begin{definition} A {\em polarization} of $L$ or $\Lambda$ is a transverse Lagrangian distribution. \end{definition}

\begin{remark} We emphasize the transversality to an  arboreal Lagrangian means transversality to its {\em closed} smooth  pieces, and not just to open strata. \end{remark}

Before we continue we introduce some auxiliary notions. Let $V$ be a symplectic vector space and $\ell_1,\ell_2,\ell_3 \subset V$ linear Lagrangian subspaces which are pairwise transverse. We write $\ell_1 \prec \ell_2 \prec \ell_3$ if $\ell_3$ corresponds to a positive definite quadratic form with respect to the polarization $(\ell_1,\ell_2)$ of $V$. Let $C \subset  V$ be a coisotropic subspace. For any linear Lagrangian subspace $\ell \subset V$ we denote by $[\ell]^C$ the symplectic reduction of $\ell$ with respect to $C$.

  Let $L$ be an arboreal Lagrangian whose germ at a  point $\lambda\in L$ has the type $(\cT=(T,\rho,\eps),m)$. Let $L_\rho\subset T_\lambda X$ the tangent plane to the root Lagrangian corresponding to the root $\rho$. For each vertex $\alpha$ connected by an edge with $\rho$ let $L_\alpha\subset T_\lambda X$ denote the Lagrangian plane tangent to the Lagrangian  corresponding to the vertex $a$.  We recall that $L_\rho$ and $L_\alpha$ cleanly  intersect along a codimension $1$ subspace.
Consider a coistropic subspace  $C_\alpha:=\Span(L_\rho, L_\alpha)\subset T_\lambda X $. Let $\eta$ be a Lagrangian distribution in $X$ transverse to  $L$.
Define the sign
\begin{equation}\label{eq:arb-sign}
\eps(\eta,L,\alpha)=\begin{cases}
 +1,&\hbox{if}\;\;
 [L_\rho]^{C_\alpha}\prec [L_\alpha]^{C_\alpha}\prec  [\eta]^{C_\alpha};\\
 -1,&\hbox{if}\;\;
 [L_\rho]^{C_\alpha}\prec [\eta]^{C_\alpha}\prec  [L_\alpha]^{C_\alpha}.
 \end{cases}
\end{equation}

    \begin{figure}[h]
\includegraphics[scale=0.6]{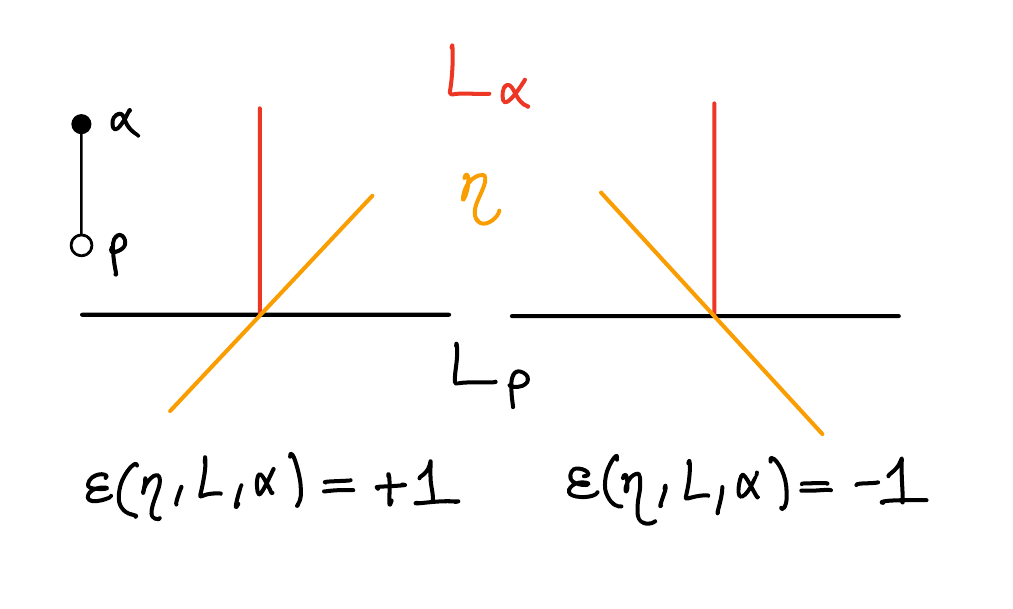}
\caption{The notion of sign for the $A_2$ singularity.}
\label{fig:polarsign}
\end{figure}

Similarly, if $\Lambda$ is an arboreal Legendrian in a contact manifold $(Y,\xi)$, and $\eta$ a Legendrian distribution transverse to $\Lambda$, then for any point $\lambda\in\Lambda$ of type $\cT=(T,\rho,\eps)$ we assign a sign $ \eps(\eta,\Lambda,\alpha)$ for every vertex $\alpha$ adjacent  to the root $\rho$ as equal to $\pm 1$  depending on the $\prec$-order of the triple $ [L_\rho]^{C_\alpha}, [L_\alpha]^{C_\alpha},  [\eta]^{C_\alpha}$ in $[\xi_\lambda]^{C_\alpha}$.

%%%%%%%
  
% 
%\begin{lemma}\label{lm: can models are arb}
%Fix a signed rooted tree $\cT = (T, \rho, \eps)$, and set $n = |n(\cT)|$.
%Let $\rho_1,\dots, \rho_k$ be the vertices adjacent to $\rho$, and set $\cT_j \subset \cT \setminus \rho$, for $j=1,\dots, k$, to be the signed rooted subtree with $\rho_j$ as root.
%Then there are disjoint arboreal Legendrians $\Lambda_j \subset J^1\R^{n}$ of type $(\cT_j, n)$ such $H_\cT=  \bigcup_{j = 1}^k\pi(\Lambda_{\cT_j})$ as fronts in $\R^{n+1}$.
% \end{lemma}
%
%
%\begin{proof}
%Let us first assume that there is only one root $\rho_1$ adjacent to $\rho$. In that case
% 
% 
% Consider a model Lagrangian $L_\cT\subset T^*\R^n$ and recall that $J^1\R^{n-1}$ can be contactly embedded to  the space $\p_\infty T^*\R^n$   of cooriented contact elements as the subspace of hyperplanes transverse to the $q_0$-axis with the coorientation by the vector field $\frac{\p}{\p q_0}$. The Legendrian ideal boundary $\Lambda=\p_\infty L_\cT=(L_\cT\setminus L_\rho)/\R$   is contained in $J^1\R^{n-1}$ and is a   disjoint union of  Legendrians $\Lambda_1,\dots, \Lambda_k$ enumerated by the  vertices $\rho_1,\dots, \rho_k$ of the graph $\cT$ adjacent to the root $\rho$. Applying to each component the transformation 
% of Lemma \ref{lem: tilt signs}
% \footnote{Adjust Lemma \ref{lem: tilt signs}} we get the models $L_{\cT_j}\subset T^*\R^{n-1}\times 0\subset
%J^1\R^{n-1}$.
%\end{proof}
%%%%%%
  
  \subsection{Stability of arboreal Lagrangians and Legendrians}
  
  The following is the main result of Section \ref{sec:stab}. We use below the notation $\ssT^*M$ for the germ of the cotangent bundle $T^*M$ along $M$.

 \begin{theorem}\label{thm:unique-signed}
 Let $\cT$ be a signed rooted tree. Let $\rho_1,\dots,\rho_k$ be vertices adjacent to the root $\rho$ and $\cT_j$ be subtrees  with roots $\rho_j$ (where we removed the decoration of edges $[\rho_j\alpha]$). Let $\phi_j:
  \ssT^*\R^m\to J^1\R^m$, $m\geq n=n(\sT)$, be  germs of Weinstein hypersurface embeddings with disjoint images.  Denote $z_j:=\phi_j(0)$, $\Lambda^j=\phi_j(L_{\cT_j}\times\R^{m-n(\cT_j)})$,  $j=1,\dots, k$.  
  Suppose that 
 \begin{itemize}
 \item[(i)] $\pi(z_j)=0$;
 \item[(ii)]  the arboreal Legendrian $\Lambda:=\bigcup_{j=1}^k\Lambda^j$ projects transversely under the front projection $J^1\R^n\to\R\times \R^n$;
 \item[(iii)] for each edge $[\rho_j\alpha]$  we have $\eps(\nu,\Lambda^j,\alpha)=\eps_{[\rho_j\alpha]}$.
 \end{itemize}
   Then $\R^m\cup C(\Lambda)$, where $C(\Lambda)$ is the Liouville cone of $\Lambda$, is an  arboreal Lagrangian of type $(\cT,m)$ or equivalently, the  germ of the front $\pi(\Lambda)$ is diffeomorphic to $H_\cT\times \R^{m-n(\cT)}$.   \end{theorem}
Theorem \ref{thm:unique-signed} is a corollary of its unsigned version which is the content of the following proposition.
  \begin{proposition}\label{prop:unique-unsigned}
   Let $\sT$ be a   rooted tree. Let $\rho_1,\dots,\rho_k$ be vertices adjacent to the root $\rho$ and $\sT_j$ be subtrees  with roots $\rho_j$. Let $\phi_j:
  \ssT^*\R^m\to J^1\R^m$, $m\geq n=n(\sT)$, be   germs of Weinstein hypersurface embeddings.  Denote $z_j:=\phi_j(0)$, $  \Lambda^j=\phi_j( L_{\sT_j}\times\R^{m-n(\sT_j)})$,  $j=1,\dots, k$.  
  Suppose that 
 \begin{itemize}
 \item[(i)] $\pi(z_j)=0$;
 \item[(ii)]  the extended arboreal Legendrian $\Lambda:=\bigcup_{j=1}^k\Lambda^j$ projects transversely under the front projection $J^1\R^n\to\R\times \R^n$;
 \end{itemize}
   Then $\R^m\cup C(\Lambda)$ is an  extended arboreal Lagrangian of type $(\sT,m)$, or equivalently, the  germ of the front $\pi(\Lambda)$ is diffeomorphic to $H_\sT\times \R^{m-n(\sT)}$. 
    \end{proposition}
    
    \begin{proof}[Proof of Theorem \ref{thm:unique-signed} using Proposition \ref{prop:unique-unsigned}]
    Consider the arboreal Legendrian  as a closed subcomplex of the extended model. Apply 
    Proposition \ref{prop:unique-unsigned} to assume the extended front is in canonical form. Then
    Lemma~\ref{lem: ext comp tree} implies the front of the original arboreal Legendrian is a canonical model.
    \end{proof}
    
    Proposition \ref{prop:unique-unsigned} will be proven below in this section (see Section \ref{sec:proof-cone-arb}
) below, but first we discuss some corollaries of Theorem \ref{thm:unique-signed}.
 
 %%%%%%%
\begin{cor}\label{cor:cone-over-arboreal}
Let $ \Lambda\subset \p_\infty  T^*M$ be an arboreal Legendrian. Suppose that the front  projection $\pi:\Lambda\to M$ is a transverse immersion.
Then $L:=C(\Lambda)\cup M$ is an arboreal Lagrangian. \end{cor}

    \begin{figure}[h]
\includegraphics[scale=0.65]{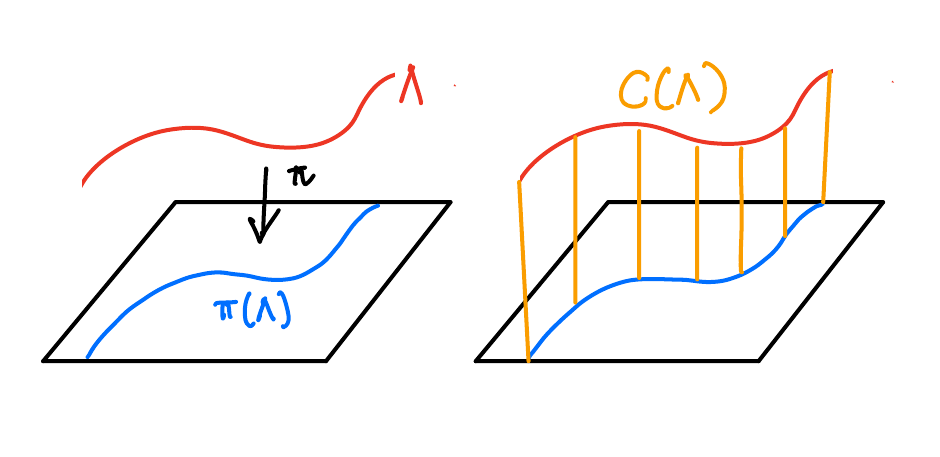}
\caption{In particular, the zero section union the Liouville cone on a regular Legendrian is arboreal with $A_2$ singularities along its front. }
\label{fig:transverseliouville}
\end{figure}

\begin{proof} The intersection $H:=M\cap \ol {C(\Lambda)}$ is the front of the Legendrian $\Lambda$. Each point $a\in H$  has finitely many  pre-images $z_1,\dots, z_k\in \Lambda$. The germs $\Lambda^j$ of $\Lambda$ at $z_j$ by  our assumption are images of arboreal Lagrangian models under Weinstein embeddings of their symplectic neighborhoods. Hence, by Theorem  \ref{thm:unique-signed} the germ of $L$ at $z$ is of arboreal type.
\end{proof}
       
  It is not a priori clear that  even  the standard Lagrangian  (resp. Legendrian) arboreal models are arboreal Lagrangians (resp. Legendrians).
  However, the following corollary shows that they are.
  \begin{cor}\label{cor:off-center-germ}
  Consider a model Lagrangian $L_\cT\subset  T^*\R^{n}, n=n(\cT)$.
  Then for any point $\lambda\in  L_\cT$ the germ of $L_\cT$ at $\lambda$ is a $(\cT',n)$-Lagrangian for a signed rooted tree $\cT'$.
  \end{cor}
  \begin{proof}
  We argue by induction in $n$. The base of the induction is trivial. Assuming the claim for $n-1$ we 
   recall that $ L_\cT$ can be presented as $L_\rho\cup C(\Lambda) $, where $L_\rho$ is the smooth piece corresponding to the root $\rho$ of $\cT$  and  $\Lambda$  is a union of  model Legendrians of dimension $n-1$ in $\p_\infty T^*(\R^n)$.  By the  induction hypothesis $\Lambda$ is an arboreal Legendrian,  and hence  applying Corollary \ref{cor:cone-over-arboreal}  we  conclude  that  $L_\cT$ is an arboreal Lagrangian.
  \end{proof}
  
  \begin{remark}
  We will not need it in what follows, so only briefly comment here that it is possible to specify precisely the type $(\cT',n)$ of the germ of $L_\cT$ at each point $\lambda\in L_\cT$. Following~\cite{N13} the  underlying tree $T'$ is a canonically defined subquotient of $T$, in other words, a diagram $T' \leftarrow S \to T$, where $S\to T$ is a full subtree, and $S\to T'$ contracts some edges; conversely, any such subquotient can occur. Furthermore, if we partially order $T$ with the root $\rho \in T$ as minimum, then the root $\rho'\in T'$ is the unique minimum of the natural induced partial order on $T'$. Finally, to equip  $T'$ with signs, we restrict the signs of $T$ to the subtree $S$, then push them forward to $T'$ using that each edge of $T'$ is the image of a unique edge of $S$.   
  \end{remark}
  
  \begin{cor}\label{cor:2-polar}
  Let $L_\cT\subset T^*\R^n$ be a model Lagrangian associated with a signed rooted tree $(T,\rho,\eps)$. Let $\eta_0,\eta_1$ be two  polarizations transverse to $L_\cT$.
  Suppose that for any vertex $\alpha$ of $T$ adjacent to $\rho$
  we have $$\eps(\eta_0,L,\alpha)=\eps(\eta_1,L,\alpha).$$  
    Then there is a (germ at the origin of)   a symplectomorphism 
  $\psi:T^*\R^n\to T^*\R^n$ such that $\psi(L)=L$ and $d\psi(\eta_0)=\eta_1$ along $L$. 
   \end{cor}
   \begin{proof} There exist embeddings $h_0,h_1:T^*\R^n\to J^1\R^n$ as Weinstein hypersurfaces, such that $h_j(\eta_j)=\nu_0$, $j=0,1$, where $\nu_0$ is the canonical Legendrian foliation of $J^1\R^n$ by fibers of the front projection to $\R^n\times\R$.
   Consider the arboreal Lagrangians $\ol L_j:=C(h_j(L_\cT))\cup(\R^n\times\R)$, $j=0,1$, and note that their arboreal types are described by the same signed rooted tree $\ol\cT$ obtained from $\cT$ by adding a new root, connecting it by an edge to the old one, and assigning to edges $[\rho\alpha]$ of $\cT\subset\ol\cT$ adjacent to the old root $\rho$ the sign $\eps(\eta_0,L,\alpha)=\eps(\eta_1,L,\alpha)$.
   Applying Theorem \ref{thm:unique-signed} we find the required symplectomorphism $\psi$.
   \end{proof}

  \begin{cor}
\label{cor:reducing-arboreal}
Let $H\subset M$ be an arboreal front. Then for any  submanifold $\Sigma\subset M$ transverse to (all strata of) $H$ the intersection $\Sigma\cap H$ is an arboreal front in $\Sigma$.
 \end{cor}

\begin{proof}
We can assume that $H$ is an arboreal front germ at a point $x\in H$, and hence the germ of $(M, H)$ at $x$  is diffeomorphic to the germ of $(\R^{n(\cT)+1}\times\R^k, H_\cT\times\R^k)$ for some rooted signed arboreal tree $\cT$ and $k=n-n(\cT)$. Note that  the transversality of $\Sigma$ to $H$ implies that $\codim \Sigma\leq k$ and that the projection of $p:\Sigma\subset \R^{n(\cT)+1}\times\R^k \to \R^{n(\cT)+1}$  to the first factor is a submersion, and because we are dealing with germs, it  is a trivial fibration. On the other hand, the projection $p|_{\Sigma\cap H}:\Sigma\cap H\to H_\cT$ is  the restriction of this fibration to $H_\cT\subset\R^{N(\cT)}$.
\end{proof}

          \begin{figure}[h]
\includegraphics[scale=0.55]{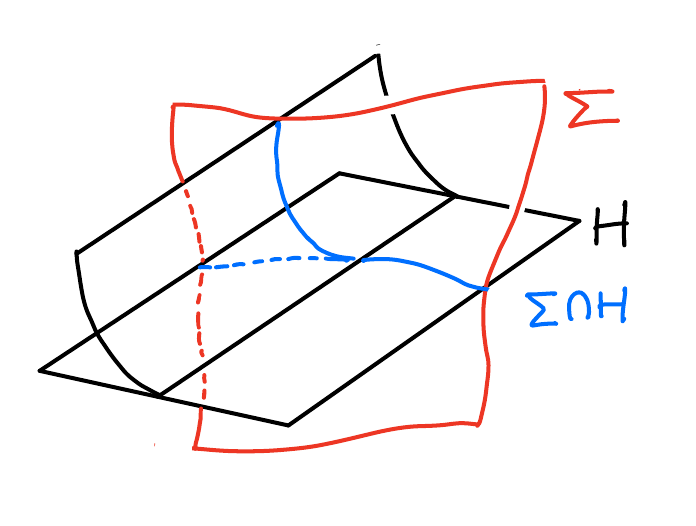}
\caption{Illustration that $\Sigma \cap H$ is an arboreal front in $\Sigma$.}
\label{fig:intersection}
\end{figure}

\subsection{Parametric version}

The following is the parametric version of Theorem \ref{thm:unique-signed}.

 \begin{theorem}\label{thm:unique-signed-parametric}
 Let $\cT$ be a signed rooted tree. Let $\rho_1,\dots,\rho_k$ be vertices adjacent to the root $\rho$ and $\cT_j$ be subtrees  with roots $\rho_j$ (where we removed the decoration of edges $[\rho_j\alpha]$). Let $\phi^y_j:
  \ssT^*\R^m\to J^1\R^m$, $m\geq n=n(\sT)$, be families of germs of Weinstein hypersurface embeddings with disjoint images, parametrized by a manifold $Y$.   Denote $z^y_j:=\phi^y_j(0)$, $\Lambda_y^j=\phi^y_j(L_{\cT_j}\times\R^{m-n(\cT_j)})$,  $j=1,\dots, k$.  
  Suppose that 
 \begin{itemize}
 \item[(i)] $\pi(z^y_j)=0$;
 \item[(ii)]  the arboreal Legendrian $\Lambda_y:=\bigcup_{j=1}^k\Lambda_y^j$ projects transversely under the front projection $J^1\R^n\to\R\times \R^n$;
 \item[(iii)] for each edge $[\rho_j\alpha]$  we have $\eps(\nu,\Lambda^j_y,\alpha)=\eps_{[\rho_j\alpha]}$.
 \end{itemize}
 Then there exists a family of diffeomorphisms $\phi_y$ between $H_\cT\times \R^{m-n(\cT)}$ and the front $\pi(\Lambda_y)$. If $K \subset Y$ is a closed subset and the $\phi_j^y$ are the standard embeddings of the local model for $y \in \Op(K)$, then we may further assume $\phi_y=\text{Id}$ for $y \in \Op(K)$. \end{theorem}
 
 The parametric version of Proposition \ref{prop:unique-unsigned} is formulated similarly. As a consequence of Theorem \ref{thm:unique-signed-parametric} we get the following result:

\begin{cor}\label{cor: contr auts}
Fix a signed rooted tree $\cT = (T, \rho, \eps)$, set $n = |n(\cT)|$ and consider  the arboreal $\cT$-front  $H_\cT \subset \R^n$. Let $D(\R^n, H_\cT)$ be the group of germs at $0$  of diffeomorphisms of $\R^n$ preserving $H_\cT$ as a front, i.e.~ as a subset along with its coorientation.

Then the fibers of the natural map $D(\R^n, H_\cT) \to \Aut(\cT)$ are weakly contractible.
\end{cor}

\begin{proof} 
We  deduce Corollary~ \ref{cor: contr auts} from Theorem \ref{thm:unique-signed-parametric}.
We will argue for  $\cT = \cA_{n+1}$ when $H_{\cA_{n+1}} = {}^{n-1} \Gamma$; the case of general $\cT$ is similar.

Since $ \Aut(\cA_{n+1})$ is trivial, we seek to show $D(\R^n, {}^{n-1} \Gamma)$ is weakly contractible. 
Note any $\varphi \in D(\R^n, {}^{n-1} \Gamma)$ preserves $0$, and moreover, preserves the canonical flag in $T_0 \R^n$ given by the tangents to the intersections $\bigcap_{i< i_0} {}^{n-1} \Gamma_i$.  

Let $D(\R^n)$ denote the group of germs at $0$ of diffeomorphisms of $\R^n$.
Consider a $k$-sphere of maps $f_t\in D(\R^n, {}^{n-1} \Gamma)$, $t\in S^k$. 
Since all $f_t$ preserve $0$ and the canonical flag in  $T_0 \R^n$, there exists a $k+1$-ball of diffeomorphisms $g_t\in D(\R^n)$, $t\in B^{k+1}$, extending $f_t$. Applying Theorem~\ref{thm:unique-signed-parametric} to the Weinstein hypersurface embeddings induced by $g_t$, we can find diffeomorphisms $h_t$ such that $h_t$  takes  $g_t( {}^{n-1} \Gamma)$ back to $ {}^{n-1} \Gamma$ and such that $h_t$ is the identity for $t\in S^k$. Then $h_t \circ g_t\in D(\R^n, {}^{n-1} \Gamma)$, $t\in B^{k+1}$, gives an extension of $f_t$ to the $k+1$-ball.
\end{proof}

We also formulate the parametric version of Corollary \ref{cor:2-polar}.

  \begin{cor}\label{cor:2-polar-parametric}
  Let $L_\cT\subset T^*\R^n$ be a model Lagrangian associated with a signed rooted tree $(T,\rho,\eps)$. Let $\eta^y_0,\eta^y_1$ be two families of polarizations transverse to $L_\cT$ parametrized by a manifold $Y$.
  Suppose that for any vertex $\alpha$ of $T$ adjacent to $\rho$
  we have $$\eps(\eta^y_0,L,\alpha)=\eps(\eta^y_1,L,\alpha).$$  
    Then there is a family of (germ at the origin of)  symplectomorphisms
  $\psi^y:T^*\R^n\to T^*\R^n$ such that $\psi^y(L)=L$ and $d\psi^y(\eta^y_0)=\eta^y_1$ along $L$. Moreover, if $\eta_0^y=\eta_1^y$ for $y \in \Op(K)$ for $K \subset Y$ a closed subset, then we can take $\psi^y=\text{Id}$ for $y \in \Op(K)$.
   \end{cor}
The proof is just like in the non-parametric case, but applying Theorem \ref{thm:unique-signed-parametric} instead of Theorem \ref{thm:unique-signed}.

%%%%%%%%%%%%%%%%%%%%%%%%%%%%%%%%%%%%%%%%%%%%%%%%%%%%%%%%%%%%%%%%%%%%%%%%%%%%%%%

\subsection{Tangency loci}

 Before proving Proposition \ref{prop:unique-unsigned} and its parametric analogue we need to analyze more closely the geometry of hypersurfaces forming arboreal fronts.
\begin{definition}
Given smooth hypersurfaces $X_1,X_2  \subset  \R^{n+1}$, we denote by $T(X_1, X_2) \subset \R^{n+1}$ their {\em tangency locus}, i.e. the subset of points $x\in X_1\cap X_2$ such that $T_x X_1 = T_x X_2$.
\end{definition}

\begin{remark}
Given smooth Legendrians $L_1, L_2 \subset J^1 \R^{n}$ whose fronts $X_1 = \pi(L_1), X_2 = \pi(L_2) \subset \R^{n+1}$ are smooth hypersurfaces, note that
 $T(X_1, X_2) = \pi(L_1\cap  L_2)$.
\end{remark}

For $0 \leq j < i \leq n$, recall the notation
$$
\xymatrix{
h_{i, j} := h_{ i - j}(x_{j+1}, \ldots, x_i)
}
$$
so in particular $h_{i, 0} = h_i(x_1, \ldots, x_i)$ and $h_{i, i-1} = h_1(x_i) = x_i$.
Set 
$$
\xymatrix{
T_{i, j} = \{h_{i, j}  = 0\} \subset \R^{n+1}
}
$$
Note $h_{i, j}$ is independent of $x_0, \ldots, x_j$, and we have
 $$
 T_{i, j} = \R^{j+1} \times {}^{n-j - 1} \Gamma_{i - j -1}
 $$
%, and so the projection of the union $\bigcup_{i = j+1}^n {}^n T({}^n  \Gamma_i, {}^n \Gamma_j)$ is the product $\R^j \times {}^{n-j - 1} \Gamma$. 

\begin{lemma}\label{lem: tang}
For $0 \leq j < i \leq n$,
the tangency locus $T({}^n \Gamma_i, {}^n \Gamma_j) \subset \R^{n+1}$   is the intersection of
either  ${}^n \Gamma_i$ or ${}^n \Gamma_j$ with the union
$$
\xymatrix{
\{h_{i, j} = 0\} \cup \bigcup \limits_{k = 0}^{j-1} \{h_{i, k} = h_{j, k} = 0\}
= T_{i, j} \cup \bigcup \limits_{k = 0}^{j-1}(T_{i, k} \cap T_{j, k})
}
$$
%cut out inside of either ${}^n \Gamma_i$ or ${}^n \Gamma_j$ by the transverse equation $h_{i, j}  = 0.$
\end{lemma}

\begin{proof}
%Note $h_{i, j}$ is independent of $x_0$, and $h_{i, j} - x_{j+1}$ is independent of $x_{j+1}$. Thus the zero-locus of $h_{i, j}$ inside of $\R \times \R^n$ is a smooth hypersurface transverse to any graph.

Since ${}^n \Gamma_i, {}^n \Gamma_j$ are the graphs of $h_i^2, h_j^2$, the  projection of $T({}^n \Gamma_i$, ${}^n \Gamma_j) $ to the domain $\R^n$ is cut out by 
$$
\xymatrix{
h_i^2 = h_j^2 & dh^2_i =  dh^2_j 
}
$$

Note  $h_i = h_{i, 0}  = x_1 - h_{i, 1}^2$, $h_j = h_{j, 0} = x_1 - h_{j, 1}^2$. By examining the $dx_1$-component of $dh^2_i = dh^2_j $, we see it implies
$h_i = h_j$. Thus  the  projection of $T({}^n \Gamma_i$, ${}^n \Gamma_j) $ is cut out by 
the  single equation
$
dh^2_i = dh^2_j
$
which in turn
 implies $h_i = h_j$.

To satisfy 
$
dh^2_i = dh^2_j
$,  so in particular $h_i = h_j$, there are two possibilities: (i) $h_i = h_j  = 0$; or (ii) $h_i = h_j  \not = 0$.
In case (i), we find the subset $\{ h_{i, 0} = h_{j, 0} = 0\}$ appearing in the union of  the assertion of the lemma. 
In case (ii), we observe  $
dh^2_i = dh^2_j
$ 
 is then equivalent to  
$
dh^2_{i, 1} = dh^2_{j, 1}
$
which in turn implies $h_{i, 1} = h_{j, 1}$.

Now we repeat the argument. To satisfy $
dh^2_{i, 1} = dh^2_{j, 1},
$
 so in particular $h_{i, 1} = h_{j, 1}$, there are two possibilities: (i) $h_{i, 1} = h_{j, 1}  = 0$; or (ii) $h_{i, 1} = h_{j, 1}  \not = 0$. In case (i), we find the subset $\{ h_{i, 1} = h_{j, 1} = 0\}$ appearing in the union of  the assertion of the lemma. 
In case (ii), we observe
$
dh^2_{i, 1} = dh^2_{j, 1}
$
  is then equivalent to  
$
dh^2_{i, 2} = dh^2_{j, 2}
$
which in turn implies $h_{i, 2} = h_{j, 2}$.

Iterating this argument, we obtain the subset  $\bigcup_{k = 0}^{j-1} \{h_{i, k} = h_{j, k} = 0\}$, and arrive at the final equation
$
dh^2_{i, j} = 0.
$
By examining the $dx_{j+1}$-term, we see 
$
dh^2_{i, j} = 0
$
 holds if and only if $h_{i, j} = 0$, which gives the remaining subset of the assertion of the lemma.
\end{proof}

\begin{remark}
The only evident redundancy in the description of the lemma is $T_{i, j-1} \cap T_{j,j-1} \subset T_{i, j}$ since 
$h_{i, j-1} = x_j - h_{i,j}^2$, $h_{j, j-1} = x_j$, so their vanishing implies the vanishing of $h_{i,j}$.   
\end{remark}

We will be particularly interested in the locus $T_{i, j} \subset T({}^n \Gamma_i, {}^n \Gamma_j)$ and formalize its structure in the following definition.

\begin{definition}
Given smooth hypersurfaces $X_1, X_2 \subset  \R^{n+1}$, we denote by $\tau^\circ(X_1, X_2) \subset T(X_1, X_2)$ the subset of points $x\in X_1 \cap X_2$ where in some local coordinates we have  $X_1 = \{x_0  = 0\}$, $X_2= \{ x_0 = x_1^2\}$. We write   $\tau(X_1, X_2)  \subset T(X_1, X_2)$ for the closure of $\tau^\circ(X_1, X_2)$, and refer to it as the {\em primary tangency} of~$X_1, X_2$.
\end{definition}

\begin{remark}
Given smooth Legendrians $L_1, L_2 \subset J^1 \R^{n}$ whose fronts $X_1 = \pi(L_1), X_2 = \pi(L_2) \subset \R^{n+1}$ are smooth hypersurfaces, note that
 $\tau^\circ(X_1, X_2)$ is the front projection
of where $L_1, L_2$ intersect cleanly in codimension one.
\end{remark}

We have the following consequence of Lemma~\ref{lem: tang}.

\begin{cor}\label{cor: quad prim tang}
For $0 \leq j < i \leq n$,
the primary tangency $\tau({}^n \Gamma_i, {}^n \Gamma_j) \subset \R^{n+1}$   is the intersection of 
either  ${}^n \Gamma_i$ or ${}^n \Gamma_j$
with $T_{i, j}$.
\end{cor}

Before continuing, let us record the following for future use.

\begin{lemma}\label{lem: tang of tang}
 Fix $0 \leq k < j\leq n-1$. 
 
We have 
$$
 \xymatrix{
 \tau(\tau(\Gamma_n,{}^n \Gamma_k),\tau({}^n \Gamma_j,{}^n \Gamma_k))=\tau ({}^n \Gamma_n,{}^n \Gamma_j)  \cap \tau({}^n \Gamma_j,{}^n \Gamma_k) 
 }$$
 where the
 primary tangency of $\tau(\Gamma_n,{}^n \Gamma_k)$, $\tau({}^n \Gamma_j,{}^n \Gamma_k)$ of the left hand side is calculated in ${}^n \Gamma_k \simeq \R^n$.
 
\end{lemma}

         \begin{figure}[h]
\includegraphics[scale=0.6]{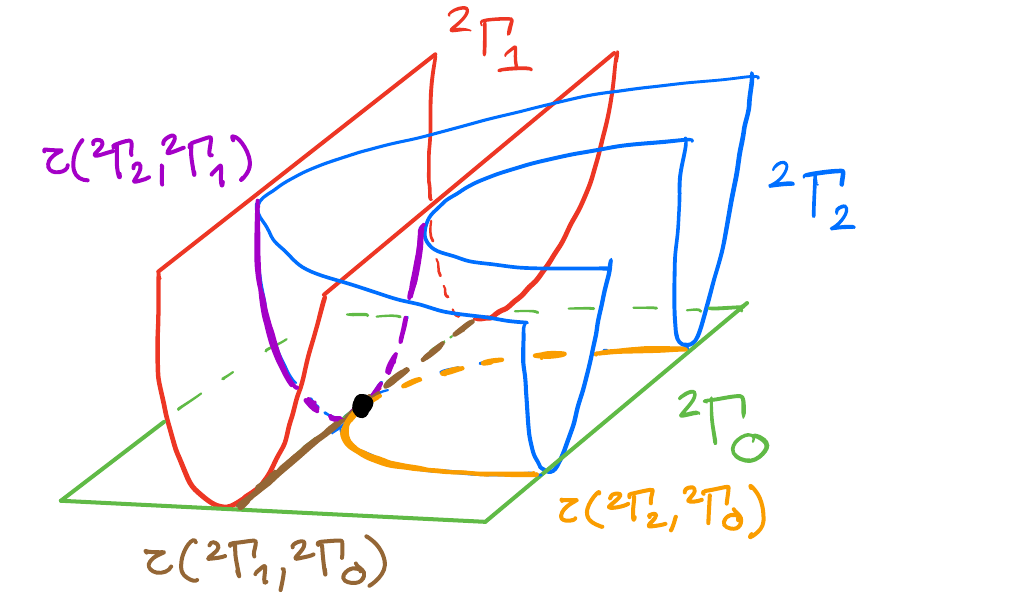}
\caption{Verification of the conclusion of Lemma \ref{lem: tang of tang} for $n=2$, in this case both the right and left hand sides of the equality $\tau(\tau({}^2 \Gamma_2, {}^2 \Gamma_0), \tau({}^2\Gamma_1,{}^2\Gamma_0))=\tau({}^2 \Gamma_2,{}^2\Gamma_1) \cap \tau({}^2 \Gamma_1 ,{}^2 \Gamma_0)$ consist of the origin.}
\label{fig:tangencies}
\end{figure}

\begin{proof}
By the preceding corollary, the left hand side is  the intersection ${}^n \Gamma_k \cap \tau(T_{n, k} , T_{j, k})$.

%Note $T_{n, k} , T_{j, k}$, and hence  also $\tau(T_{n, k} , T_{j, k})$, are independent of $x_0$. 
Note  ${}^n \Gamma_k  \cap T_{j, k}  = \tau({}^n \Gamma_j, {}^n \Gamma_k) =  {}^n \Gamma_j \cap T_{j, k}$. 
Hence  
$${}^n \Gamma_k \cap \tau(T_{n, k} , T_{j, k}) = 
{}^n \Gamma_j \cap \tau(T_{n, k} , T_{j, k})$$ 
since $y \in {}^n \Gamma_k \cap \tau(T_{n, k} , T_{j, k})$  $\iff$ $y \in {}^n \Gamma_k \cap T_{j, k}$, $y\in \tau(T_{n, k} , T_{j, k})$ $\iff$
 $y \in {}^n \Gamma_j \cap T_{j, k}$, $y\in \tau(T_{n, k} , T_{j, k})$
 $\iff$ $y \in {}^n \Gamma_j \cap \tau(T_{n, k} , T_{j, k})$.

Next, recall  
 $$
 \xymatrix{
 T_{n, k} = \R^{k+1} \times {}^{n-k - 1} \Gamma_{n - k -1}
 &
 T_{j, k} = \R^{k+1} \times {}^{n-k - 1} \Gamma_{j - k -1}
 }
 $$
  Hence by the preceding corollary, we have 
 $$
 \xymatrix{
 \tau(T_{n, k} , T_{j, k}) = T_{j, k} \cap \{h_{n,j} = 0\} %= \{h_{j, k} = h_{n,j} = 0\} 
 }
 $$
 Thus the left hand side is  given by $ {}^n \Gamma_j \cap T_{j, k} \cap T_{n, j}$.

On the other hand, by the preceding corollary, the right hand side is also given by
$
{}^n\Gamma_j \cap T_{n, j} \cap T_{j, k}. % = \{h_{n, j} = h_{j, k} = 0\}
$
\end{proof}

%%%%
\subsubsection{More on distinguished quadrants}

\begin{cor}\label{cor: embed}
For $0 \leq j < i \leq n$, we have 
$$
\xymatrix{
{}^n \Gamma^\eps_i \cap {}^n \Gamma^\eps_j  = T({}^n \Gamma^\eps_i, {}^n \Gamma^\eps_j) = \tau({}^n \Gamma^\eps_i, {}^n \Gamma^\eps_j)
}
$$
and they 
coincide with the closed boundary face of ${}^n \Gamma^\eps_i$ cut out by
$
h_{i, j}= 0.
$
\end{cor}

\begin{proof}
For $j = 0$, we have ${}^n \Gamma^\eps_0 = {}^n \Gamma_0 = \{ x_0 = 0 \}$. From the definitions, we have 
$$
\xymatrix{
{}^n \Gamma^\eps_i \cap {}^n \Gamma_0  = T({}^n \Gamma^\eps_i, {}^n \Gamma_0) = \tau({}^n \Gamma^\eps_i, {}^n \Gamma_0)
}
$$
which is cut out of ${}^n P^\eps_i$ by $
h_{i, 0} = h_i   = 0.$ %, and hence is a boundary face of ${}^n P^\eps_i$.

For $j>0$, the assertions follow from Lemma~\ref{lem: induction for fronts} by induction on $n$.
\end{proof}

\begin{remark}\label{rem: quad cover}
Note for any $0 \leq j < i \leq n$, we have 
$$
\xymatrix{
 \tau({}^n \Gamma_i, {}^n \Gamma_j) = \bigcup_{\eps} \tau({}^n \Gamma^\eps_i, {}^n \Gamma^\eps_j)
}
$$
To see this, consider $x\in \tau({}^n \Gamma_i, {}^n \Gamma_j)$,  so  that $h_{i, j}(x) = 0$ by Corollary~\ref{cor: quad prim tang}.   Choose $\eps$ so that $x\in {}^n \Gamma^\eps_i$. 
%Note $x$ lies in the boundary face of ${}^n \Gamma^\eps_j$ cut out by  $h_{i, j}= 0$. 
Then by Corollary~\ref{cor: embed}, we have $x\in \tau({}^n \Gamma^\eps_i, {}^n \Gamma^\eps_j)$.
\end{remark}

%
%\begin{proof}
%Immediate from Lemma~\ref{lem: induction for fronts}.
%\end{proof}
%

For $i = 0$, let ${}^{n}L^\eps_0 =   \R^{n} \subset T^*\R^{n}$ denote the zero-section.
For $i = 1, \ldots, n$,  consider the conormal bundles
$$
\xymatrix{
{}^{n} L^\eps_{i} = T^*_{ {}^{n-1}\Gamma^\eps_{i-1}} \R^{n} \subset T^* \R^{n}
}
$$
 and their union
$$
\xymatrix{
{}^{n} L^\eps = \bigcup_{i = 0}^n {}^n  L^\eps_i
}
$$

%Regard ${}^n\Gamma_i$ as cooriented by the positive $\partial_{x_0}$ direction
Similarly, for $i = 0, \ldots, n$, consider the smooth Legendrian  
$$
\xymatrix{
{}^{n} \Lambda^\eps_i \subset J^1 \R^{n}
}
$$
that maps diffeomorphically to ${}^n\Gamma^\eps_i \subset \R^{n+1}$ under the front projection $\pi:J^1 \R^n \to  \R^{n+1}$,
and their union
$$
\xymatrix{
{}^n  \Lambda^\eps = \bigcup_{i =0}^n {}^{n}  \Lambda^\eps_i
}
$$

Note the contactomorphism of Lemma~\ref{lem: tilt}
takes ${}^{n}  \Lambda^\eps_i \subset J^1 \R^{n}$  isomorphically to
$\{0\} \times {}^n L^\eps_i \subset \{0\} \times T^*\R^n$, and
thus  ${}^{n}  \Lambda^\eps \subset J^1 \R^{n}$  isomorphically to
$\{0\} \times {}^n L^\eps \subset \{0\} \times T^*\R^n$.

We have the following topological consequence of  Lemma~\ref{lem: induction for fronts}.

\begin{cor} 
As a union of smooth manifolds with corners, ${}^n \Gamma^\eps \subset \R^{n+1}$ is given by the gluing
$$
\xymatrix{
{}^n \Gamma^\eps = ({}^{n-1} \Gamma^{\eps'} \times \R_{\geq 0}) \coprod_{({}^{n-1} \Gamma^{\eps'} \times \{0\})} (\R^n \times \{0\})
}
$$
where $\eps' = (\eps_0\eps_1, \eps_2, \ldots, \eps_n)$. The front projection takes ${}^n L^\eps\subset J^1 \R^{n}$ homeomorphically to 
${}^n \Gamma^\eps\subset \R^{n+1}$.
\end{cor}

%%%%

Before continuing, let us record the following for future use.

\begin{cor}\label{cor: clean intersect}
For $0 < j < i \leq n$, the closure of the codimension one clean intersection of ${}^nL^\eps_i, {}^n L_j$
is precisely $ {}^n L^\eps_i \cap {}^n L^\eps_j$.
\end{cor}

\begin{proof}
The  closure of the codimension one clean intersection of ${}^n L^\eps_i, {}^n \Lambda_j$ 
is conic and projects to the primary tangency of ${}^{n-1} \Gamma^\eps_{i-1}, {}^{n-1} \Gamma_{j-1}$.
By Corollary~\ref{cor: quad prim tang}, the primary tangency of ${}^{n-1} \Gamma_{i-1}, {}^{n-1} \Gamma_{j-1}$
is cut out by $h_{i-1,j-1} = 0$. By Corollary~\ref{cor: embed}, this is precisely the tangency 
$T({}^{n-1} \Gamma^\eps_{i-1}, {}^{n-1} \Gamma_{j-1})$ and hence lifts precisely to the conic intersection 
$ {}^n L^\eps_i \cap {}^n L^\eps_j$.
\end{proof}

%%%%
 
 \subsection{The case of $\sA_{n+1}$-tree}
 The following Theorem \ref {thm: quad unique} will play a key   role    in proving Proposition \ref{prop:unique-unsigned}.

\begin{theorem}\label{thm: quad unique}

Let $\varphi:T^*\R^n\to J^1 \R^n$ be an embedding as a Weinstein hypersurface. Assume  that the image of ${}^nL$ under $\varphi$ is transverse   to the fibers of the projection $J^1\R^n\to  \R^n$.  Let $\Upsilon = \pi (\varphi({}^nL)) \subset \R \times \R^n$ be (the germ of) the front at the central point.

Then there exists a diffeomorphism $\R\times\R^n\to\R\times \R^n$ taking $\Upsilon$ to the germ at the origin of ${}^n \Gamma \subset \R \times \R^n$.

\end{theorem}

The proof of Theorem~\ref{thm: quad unique} will proceed by induction on the dimension $n$. At each stage, we will prove the fully parametric version:

\begin{theorem}\label{thm: quad unique parametric}

Let $\varphi^y:T^*\R^n\to J^1 \R^n$ be a family of  Weinstein hypersurface embeddings parametrized by a manifold $Y$. Assume  that the image of ${}^nL$ under $\varphi^y$ is transverse   to the fibers of the projection $J^1\R^n\to  \R^n$.  Let $\Upsilon^y = \pi (\varphi^y({}^nL)) \subset \R \times \R^n$ be (the germs of) the fronts at the central points.

Then there exists a family of diffeomorphisms $\psi^y:\R\times\R^n\to\R\times \R^n$ taking $\Upsilon^y$ to the germ at the origin of ${}^n \Gamma \subset \R \times \R^n$. If $\varphi^y=\text{Id}$ for $y \in \Op(K)$, where $K \subset Y$ is a closed subset, then we may assume $\psi^y=\text{Id}$ for $y \in \Op(K)$.

\end{theorem}

As usual the case of general pairs $(Y,K)$ follows from the case $Y=D^k$ and $K=S^{k-1}$. 

\subsubsection{Base case $n=0$}\label{s: n=0} The  $k$-parametric version states: the germ of any graphical hypersurface $\Upsilon \subset  \R \times \R^k$ is diffeomorphic to the germ of the zero-graph ${}^0 \Gamma \times \R^k  = \{0\} \times \R^k$. This can be achieved by an isotopy generated by a time-dependent vector field of the form $h_t \partial_{x_0}$. This vector field is zero at infinity if $\Upsilon$ is standard at infinity.
%
%{\bf Base of induction. The case $n=1$.}
%Let $\wt\Gamma_0,\Gamma_1$ be a general arboreal model in $\R^2$.
%First, we can normalize $\wt\Gamma_0$ to coincide with $\Gamma_0$.
%By the definition of an arboreal model  $\wt\Gamma_1$ has a unique non-degenerate tangency to $\Gamma_0$ at a unique point.
%By an isotopy preserving the coordinate $q_0$ we can move the tangency point to $0$ and then by an isotopy generated by a vector field of the form
%$q_1\alpha(q)\frac{\p}{\p q_1}$ (and hence preserving $\Gamma_0$ we can move $\wt\Gamma_1$ to $\Gamma_1$. The whole process works parametrically.
%\medskip

\subsubsection{Case $n=1$}\label{s: n=1}
The next case of the induction $n=1$ is elementary but slightly different from the others, so it is more convenient to treat  separately. 

With the setup of the theorem, consider the front 
$\Upsilon = \pi( {}^1 \Lambda) \subset \R^2$, and assume without loss of generality
that the origin is the central point. By induction, we may assume, the front takes the form
$\Upsilon =  \Gamma_0  \cup \Upsilon_1 \subset \R^2$ where $ \Gamma_0 = \{x_0 = 0\}$.
Near the origin, the intersection $\Gamma_0\cap  \Upsilon_1$ and tangency locus $T(\Gamma_0, \Upsilon_1)$ coincide and consist of the origin alone. Moreover, by construction, the origin is a simple tangency, and so  $\Upsilon_1 = \{ x_0 = \alpha x_1^2\}$ with $\alpha(0) \not = 0$.
Now it is elementary to find a time-dependent vector field of the form $h_t x_1 \partial_{x_1}$, hence vanishing on $\Gamma_0$, generating an isotopy taking
$\Upsilon_1$ to either $ \Gamma_1 = \{ x_0 =  x_1^2\}$ or $- \Gamma_1 = \{ x_0 =  -x_1^2\}$. In the former case, we are done;  in the latter case, we may apply the diffeomorphism $(x_0, x_1) \mapsto (-x_0, x_1)$ to arrive at the configuration $\Gamma_0  \cup \Gamma_1$. Finally, it is evident the prior constructions can be performed parametrically, with the vector field zero at infinity if $\Upsilon$ is standard at infinity.

\subsubsection{Inductive step} \label{s: inductive step}

The inductive step takes the following form. Suppose the  fully parametric assertion has been established for dimension $n-1$. Starting from ${}^{n} \Lambda  \subset T^*\R^n$, remove the last smooth piece to obtain ${}^n \Lambda' = {}^{n} \Lambda \setminus  {}^{n} \Lambda_n$, and consider the
 corresponding front  $\Upsilon' = \pi({}^n \Lambda')$. Note that ${}^n \Lambda' = {}^{n-1} \Lambda \times \R \subset T^*(\R^{n-1} \times \R)$, and so by an inductive application of the 1-parametric version of the theorem, we may assume 
 $$
 \xymatrix{
 \Upsilon' =   {}^{n-1} \Gamma \times \R 
 }
 $$ 
  
  Set $\Upsilon_n = \pi({}^n \Lambda_n)$.
  We will find a diffeomorphism $ \R^{n+1}\to \R^{n+1}$ that preserves $ \Upsilon' $ (as a subset, not pointwise), and takes $\Upsilon_n$ to ${}^n \Gamma_n$. Moreover, it will be evident the diffeomorphism can be constructed in  parametric form, including the relative parametric form.  This will complete the inductive step and prove the theorem.

\subsubsection{Two propositions}
 The proof of the inductive step is based on the following 2 propositions.

  \begin{prop}\label{prop:divisibility}
  Fix $n\geq 2$.
  
With the setup of Theorem~\ref{thm: quad unique}, suppose 
  $
  \Upsilon = \bigcup_{i = 0}^{n-1} {}^n \Gamma_i \cup \Upsilon_n
  $ where $\Upsilon_n = \pi({}^n \Lambda_n)$.
  Suppose in addition $ \Upsilon_n$ has primary tangency loci satisfying
  $$
  \xymatrix{
  \tau(\Upsilon_n ,{}^n \Gamma_i) \supset \tau({}^n \Gamma_n ,{}^n \Gamma_i) & i = 0, \ldots, n-1
  }
  $$
 
  Then $\Upsilon_n =\{x_0=\alpha h^2_n\}$ where
$$
\xymatrix{
    \alpha  =1+\beta \prod \limits_{j = 1}^{n-1} h^2_{n, j}
=1+\beta  h^2_{n, 1}\cdots h^2_{n, n-1} 
}
$$

Moreover, the same holds in parametric form.
 \end{prop}

    \begin{proof} 
    We have $\Upsilon_n = \{ x_0 = g\}$ for some $g$. Since $\tau(\Upsilon_n ,{}^n \Gamma_0)\supset \tau({}^n \Gamma_n ,{}^n \Gamma_0) = \{h_n = 0\}$, we must have 
    $g$ is divisible by $h_n^2$, hence $g=\alpha h^2_n$, for some $\alpha$. Next, for any $j \not = 0, n$, by Lemma~\ref{lem: tang}, 
    $\tau({}^n \Gamma_n ,{}^n \Gamma_j)$ is cut out  by  $h_{n, j} = 0$. 
     Since $\tau(\Upsilon_n ,{}^n \Gamma_j)\supset \tau({}^n \Gamma_n ,{}^n \Gamma_j)$, and
     $h_n \not = 0$ along a dense subset of $\{h_{n, j} = 0\}$, taking the ratio $g/h_n^2$ shows  that we must have $\alpha = 1 + \delta$, where $\delta$ is divisible by $h_{n,j}^2$. Repeating this argument, and using the transversality of the level-sets of the collection $h_{n, j}$, we conclude that $\delta=\beta  h^2_{n, 1}\cdots h^2_{n, n-1} $.  
         \end{proof}

    \begin{prop}\label{prop:vf}
    
     Fix $n\geq 2$.

  With the setup of Theorem~\ref{thm: quad unique}, suppose 
  $
  \Upsilon = \bigcup_{i = 0}^{n-1} {}^n \Gamma_i \cup \Upsilon_n
  $ where $\Upsilon_n = \pi({}^n \Lambda_n)$.
    Suppose in addition $\Upsilon_n =\{x_0=\alpha h^2_n\}$ where
$$
\xymatrix{
    \alpha  =1+\beta \prod \limits_{j = 1}^{n-1} h^2_{n, j}
=1+\beta  h^2_{n, 1}\cdots h^2_{n, n-1} 
}
$$

Consider the family $\Upsilon_{n,t} =  \{x_0=(1 - t + t\alpha) h^2_n\}$ so that $\Upsilon_{n,0} = {}^n \Gamma_n$, $\Upsilon_{n,1} = \Upsilon_n$. 

Then there exist functions $g_t: \R^{n+1} \to \R$ such that the vector fields 
$$
\xymatrix{
g_t v_{n-1}=  g_t \sum_{i = 0 }^{n-1} x_i  \frac 1{2^i} \partial_{x_i} = g_t x_0 \partial_{x_0} + \frac 12 g_t  x_1 \partial_{x_1} + \cdots + \frac1{2^{n-1}}g_t x_{n-1} \partial_{x_{n-1}}
}
$$
generate an isotopy $\varphi_t:\R^{n+1} \to \R^{n+1}$ such that $\varphi_t( \Upsilon_{n, 0}) = \Upsilon_{n, t}$.

In addition, the  functions $h_t$, hence vector fields $h_t v_{n-1}$, are divisible by the product
$
\prod_{j = 1}^{n-1} h_{n, j}.
$

Moreover, all of the above holds in parametric form.

    \end{prop}

The   following  lemmas  are needed for the proof of Proposition \ref{prop:vf}.

\begin{lemma}\label{lem: sym} For all $0\leq i\leq n$, the vector field 
$$
\xymatrix{
v_i =  \sum_{j = 0 }^n x_j  \frac 1{2^j} \partial_{x_j} = x_0 \partial_{x_0} + \frac 12 x_1 \partial_{x_1} + \cdots + \frac1{2^i} x_i\partial_{x_i}
}
$$ preserves each ${}^n \Gamma_j \subset \R^{n+1}$, for $j = 0, \ldots, i$.
\end{lemma}

\begin{proof}
Since ${}^n \Gamma_j \subset \R^{n+1}$ is independent of $x_{j+1}, \ldots, x_n$, it suffices to prove the case $i = j = n$.
Recall ${}^n \Gamma_n$ is the zero-locus of $f = x_0 - h_n^2$. We will show $v(h_n) = \frac 1 2 h_n$ and so $v(f) = f$. Recall $h_n = h_{n, 0} = x_1 - h_{n, 1}^2$, and in general $h_{n, j} = x_{j+1} - h_{n, j+1}^2$ with $h_{n, n-1} = x_n$. Thus $v_n(h_{n, n-1}) = \frac 1{2^n} h_{n, n-1}$, and by induction, $v(h_{n, j}) =  \frac 1{2^{j+1}} h_{n, j}$, so in particular $v(h_{n, 0}) = v(h_n) = \frac 1 2 h_n$.
\end{proof}

\begin{remark}
 In the context of the inductive step outlined above, we  will use Lemma~\ref{lem: sym} in particular the vector field
$$
\xymatrix{
v_{n-1} =  \sum_{i = 0 }^{n-1} x_i  \frac 1{2^i} \partial_{x_i} = x_0 \partial_{x_0} + \frac 12 x_1 \partial_{x_1} + \cdots + \frac1{2^{n-1}} x_{n-1} \partial_{x_{n-1}}
}
$$ 
to move $\Upsilon_n$ to ${}^n \Gamma_n$. The lemma confirms we will preserve $\Upsilon' =  {}^{n-1} \Gamma \times \R = \bigcup_{i = 0}^{n-1} {}^n \Gamma_i$.
\end{remark}

%%%%%%%%%%%%%%%%%%%%%%%%%%%%%%%%%%%%%%%%%%%%%%%%%%%%%%%%%%%%%%%%%%%%%%%%%%%%%%%%%%%%%%%%%%%%%%%%%%%%%%%

\begin{lemma}\label{lem:deriv} For any $0 \leq j < i \leq n$, and $1 \leq k \leq i$, we have 
$$
\frac{\p h^2_{i}}{\p x_k}= -(-2)^{k} \prod\limits_{j=0}^{k-1}  h_{i, j}= -(-2)^{k} h_{i, 0} h_{i, 1} \cdots h_{i, k-1}
  $$
  \end{lemma}

\begin{proof}
Recall $h_i = h_{i, 0}$ and  the inductive formulas 
$h_{i, j} = x_{j+1} - h^2_{i, j+1}
$ 
with $h_{i, i-1} = x_i$.
Thus we have
$$
\frac{\p h^2_{i, j}}{\p x_{j+1}}= 2 h_{i, j}
 \qquad \frac{\p h^2_{i, j}}{\p x_k}= -2 h_{i, j} \frac{\p h^2_{i, j+1}}{\p x_k}  \quad k > j+1 
  $$
and the assertion follows.
\end{proof}

%  
%
%Now to the proof of the proposition.
%      The case $n=1$ is straightforward, and this is the only case when we use the non-degeneracy.  
%  
%  Suppose that $n>1$.
%Let  $L_\beta\subset T^*\R^{n+1}$  be the corresponding  $A_{n+2}$-arboreal model  corresponding to the $H_{n+1}$-model $(\Gamma_0,\dots, \Gamma_{n-1}, \Gamma)$, where $\Gamma=\{q_0=(1+\beta \prod\limits_1^{n-1} g_j^{(n-j+1)} g_n\}.$

\begin{proof}[Proof of  Proposition \ref{prop:vf}.]
Suppose
  $
  \Upsilon = \bigcup_{i = 0}^{n-1} {}^n \Gamma_i \cup \Upsilon_n
  $ where $\Upsilon_n$ is the graph of 
$$
\xymatrix{
    H_\beta  =(1+\beta \prod \limits_{j = 1}^{n-1} h^2_{n, j})h_n^2
=(1+\beta  h^2_{n, 1}\cdots h^2_{n, n-1})h_n^2 
}
$$

%
%
%We will be looking for a normalizing isotopy generated by (time dependent) vector field $v$ in $\R^{n+1}$.  Then the Hamiltonian field $X_H$ on $T^*\R^{n+1}$ with the Hamiltonian
%  $H=\sum\limits_0^{n}p_jv_j$  moves the conormals to Lagrangians forming $L_\beta$ in the class of conormals.
 
 Our aim is to find  a normalizing isotopy, generated by a time-dependent vector field $v_t$, taking the graph $\Upsilon_n = \{ x_0 = H_\beta\}$ to the standard graph ${}^n \Gamma_n = \{ x_0 = h_n^2\}$, i.e.~to the graph where $\beta = 0$, while preserving $
 \bigcup_{i = 0}^{n-1} {}^n \Gamma_i$.   Thus  
for any infinitesimal deformation in the class of functions $h_\beta$, we seek a vector field $v$ realizing 
the deformation and preserving the functions $h_0, \ldots, h_{n-1}$, i.e.~we seek to solve the system 
    \begin{equation}\label{eq1}
    \begin{split}
   & \dot h_i =0, \quad i=0,\dots, n-1\\
   &\dot H_\beta=\gamma   \prod\limits_{ j = 0}^{n-1} h^2_{n,j}  = \gamma h_{n, 0}^2 \cdots h_{n, n-1}^2
       \end{split}
    \end{equation}
   where $\dot H_\beta$ denotes the derivative of $H_\beta$ with respect to $v$, and $\gamma$ is any given smooth function.
        
   Let  $\Lambda_\beta \subset T^*\R^{n+1}$ denote the conormal  to the graph  of $h_\beta$. Any vector field $v = \sum_{j = 0}^{n}v_j \p/\p_{x_j}$ on $\R^{n+1}$ extends to a Hamiltonian vector field $v_H$ on $T^*\R^{n+1}$ with Hamiltonian  $H=\sum_{j = 0}^{n}p_jv_j$. We will find $v$ deforming the graph of $h_\beta$
   by finding $H$ so that $v_H$ deforms the conormal to the graph $\Lambda_\beta$.

  In general, for a function $f :\R^n \to \R$, with graph $\Gamma_{f} =\{ x_0 = f\} \subset \R^{n+1}$, denote the conormal to the graph  by $T^*_{\Gamma_f} \subset T^* \R^{n+1}$.    With respect to the contact form
     $p_1dx_1+\dots p_ndx_n - x_{0}dp_{0}$, the conormal
     $T^*_{\Gamma_f}$ 
         is given by the generating function
    $F(x_1,\dots,x_n)=-p_{0}f(x_1,\dots, x_n),$ i.e.~ it is cut out by the equations
   \begin{align*}
    p_i &=-p_0\frac{\p f}{\p x_i}, \quad  i=1,\dots,n\\
  x_{0}  &= f(x_1,\dots, x_n)
   \end{align*}
   
    Hence given a Hamiltonian $H=\sum_{j = 0}^{n}p_jv_j$, its restriction to the conormal $T^*_{\Gamma_f}$ is given by  
    $$
    H|_{T^*_{\Gamma_f}} = p_0 v_{0}|_{x_{0}=f} - p_{0} \sum\limits_{ j=1}^n\frac{\p f}{\p x_j}v_j|_{x_{0}=f}
    $$  
    and so further restricting to $p_{0}=1$, we find the Hamilton-Jacobi equation
    $$ 
    H|_{T^*_{\Gamma_f} \cap \{ p_0 = 1\}} = v_{0}|_{x_{0}=f} -\sum\limits_{ i = 1}^n\frac{\p f}{\p x_i}v_i|_{x_{0}=f} = 
    v_0|_{x_0 = f} - \dot f
    $$
%    or in other words, 
%     $H|_{T^*_{\Gamma_f} \cap \{ p_0 = 1\}} = -\dot f$. 
%     
     
Let us apply the above to $H_\beta$ and $h_i$, for $ i = 0, \ldots, n-1$. It allows us to transform system~\eqref{eq1}
into the system
      \begin{equation}\label{eq2}
    \begin{split}
   &v_{0}(x_1,\dots, x_n, h_i)  -\sum\limits_{j=1}^n\frac{\p h_i}{\p x_j}v_j=0,\quad  i=0,\dots, n-1\\
   & v_{0}(x_1,\dots, x_n, H_\beta ) -\sum\limits_{j=1}^n\frac{\p H_\beta}{\p x_j}v_j  = \gamma   \prod\limits_{ j = 0}^{n-1} h^2_{n,j}   
       \end{split}
    \end{equation}
    Note we can  reformulate Lemma~\ref{lem: sym} from this viewpoint:  when $\beta = \gamma=0$, 
    given any function $h = h(x_1,\dots, x_n)$,
    the functions %$v_j = \frac{x_j}{2^j}h$, i.e. 
    \begin{align}\label{eq:solution}
   v_{0}=x_{0}h,
  v_1=\frac{x_1}2h, v_2=\frac{x_2} 4 h,\dots, v_{n}= \frac{x_n}{2^n}h   \end{align}
    satisfy system \eqref{eq2}.
 %      \begin{proof}
%     We have
%    \begin{align*}
%&  -\sum\limits_{i=1}^k\frac{\p g_k}{\p q_i}v_i+v_{0}(q_1,\dots, q_n, g_k)  \\
%&=h\left(\sum\limits_{j=1}^{k}(-1)^jq_j\prod\limits_{i=1}^jf_{k-i+1}^{(i)}+ (f_k)^2\right)\\
%&=f_kh\left(\sum\limits_{j=2}^{k} (-1)^jq_j\prod\limits_{i=2}^jf_{k-i+1}^{(i)}- (f^2_{k-1})^2\right)\\
%&=h f_kf_{k-1}^2\dots (f_3^{k-3})^2\left((-1)^{k-1}q_{k-1}f_{k-1}^{2}+(-1)^kq_kf_1^kf^{k-1}_2+(-1)^k(f_2^{k-1})^2\right)\\
%&=hf_kf_{k-1}^2\dots (f_2^{k-1})^2(-1)^k(-q_{k-1}+q_k^2+q_{k-1}-q_k^2)=0.
%    \end{align*}
%      \end{proof}

      Now let  us choose $v_0, v_1, \dots v_{n-1}$ as in \eqref{eq:solution} but set $v_n=0$. This will satisfy the first $n$ equations of system \eqref{eq2}, independently of $\beta, \gamma$. From hereon, we will restrict to this class of vector fields and focus on the last equation of  system \eqref{eq2}.
      
     Let us first set $\beta=0$, so that $h_\beta = h^2_n$,  and solve  system \eqref{eq2} in this case.
   Using Lemma~\ref{lem:deriv}, we can then rewrite the left-hand side of the last equation of system \eqref{eq2}    in the form  
$$
  v_{0}(x_1,\dots, x_n, h^2_n)   -\sum\limits_{j=1}^{n-1}\frac{\p h^2_n}{\p x_j}v_j\\
 =h\left(h_n^2   -\sum\limits_{j=1}^{n-1}\frac{\p h^2_n}{\p x_j}\frac{x_j}{2^j} \right)\\
  =h\left(h_n^2   +\sum\limits_{j=1}^{n-1}  (-1)^{j} x_j \prod\limits_{k=0}^{j-1}  h_{n, k} \right)\\
 %& =h\left(h_n^2 + \sum\limits_{j=1}^{n-1}(-1)^jq_j\prod\limits_{i=1}^jf_{n-i+1}^{(i)}+ (f_n)^2\right) .
$$
Using $h_n = h_{n, 0 }$, $h_{n, k} -  x_{k+1} = -  h_{n, k-1}^2$, we can inductively simplify the term in  parentheses
  \begin{align*}
  h_n^2   +\sum\limits_{j=1}^{n-1}  (-1)^{j} x_j \prod\limits_{k=0}^{j-1}  h_{n, k}
  & 
  =   h_n(h_n   - x_1 + \sum\limits_{j=2}^{n-1}  (-1)^{j} x_j \prod\limits_{k=1}^{j-1}  h_{n, k} ) \\
  & 
  =   h_n(-h^2_{n, 1} + \sum\limits_{j=2}^{n-1}  (-1)^{j} x_j \prod\limits_{k=1}^{j-1}  h_{n, k} ) \\
  & 
  =   h_nh_{n, 1} (-h_{n, 1} + x_2 +  \sum\limits_{j=3}^{n-1}  (-1)^{j} x_j \prod\limits_{k=2}^{j-1}  h_{n, k} ) \\
% & 
%  =   h_nh_{n, 1} (h_{n, 2}^2 +  \sum\limits_{j=3}^{n-1}  (-1)^{j} x_j \prod\limits_{k=2}^{j-1}  h_{n, k} ) \\
% & 
%  =   h_nh_{n, 1} h_{n, 2} (h_{n, 2} - x_3 +   \sum\limits_{j=4}^{n-1}  (-1)^{j} x_j \prod\limits_{k=3}^{j-1}  h_{n, k} ) \\
& 
\cdots\\
  & 
  =  (-1)^{n-1} h_n h_{n, 1}h_{n, 2} \cdots h_{n, n-1}   = (-1)^{n-1}  \prod\limits_{j=0}^{n-1}  h_{n, j} 
   \end{align*}
%   
%    \begin{align*}
%  h_n^2   +\sum\limits_{j=1}^{n-1}  (-1)^{j} x_j \prod\limits_{k=0}^{j-1}  h_{n, k} 
%  & 
%  =  h_n^2   - x_1 h_{n, 0} + x_2 h_{n, 0} h_{n, 1} - x_3 h_{n, 0} h_{n, 1} h_{n, 2} +  \cdots + (-1)^{n-1} x_{n-1} h_{n, 0} \cdots h_{n, n-2}\\
%  &
%   =  h_n(h_n   - x_1 +  x_2 h_{n, 1} - x_3  h_{n, 1} h_{n, 2} + \cdots + (-1)^{n-1} x_{n-1} h_{n, 1} \cdots h_{n, n-2})\\
%  & 
%  =  h_n(-h_{n, 1}^2  +  x_2 h_{n, 1}  - x_3  h_{n, 1} h_{n, 2} + \cdots + (-1)^{n-1} x_{n-1} h_{n, 1} \cdots h_{n, n-2}\\
%  & 
%  =  h_n h_{n, 1}(-h_{n, 1}   +  x_2   - x_3  h_{n, 2} + \cdots + (-1)^{n-1} x_{n-1} h_{n, 2} \cdots h_{n, n-2})\\ 
%  & 
%  =  h_n h_{n, 1}(h^2_{n, 2}   - x_3  h_{n, 2} + \cdots + (-1)^{n-1} x_{n-1} h_{n, 2} \cdots h_{n, n-2})\\ 
%  & 
%  =  h_n h_{n, 1}h_{n, 2} (h_{n, 2}   - x_3+ \cdots + (-1)^{n-1} x_{n-1} h_{n, 3} \cdots h_{n, n-2})\\ 
%  & \cdots\\
%  & 
%  =  (-1)^{n-1} h_n h_{n, 1}h_{n, 2} \cdots h_{n, n-1}   = (-1)^{n-1}  \prod\limits_{j=0}^{n-1}  h_{n, j} 
%   \end{align*}

  Thus for $\beta=0$,  the last equation of  system \eqref{eq2} reduces  to
  $$
  (-1)^{n-1}  h \prod\limits_{j=0}^{n-1}  h_{n, j} =\gamma   \prod\limits_{ j = 0}^{n-1} h^2_{n,j}   
  $$  
  and hence can be solved by
  $$
  h = (-1)^{n-1} \gamma   \prod\limits_{ j = 0}^{n-1} h_{n,j}   
  $$  
  
  Now for general $\beta$, we will similarly calculate  the left-hand side
  of the last equation of  system \eqref{eq2}. 
 To simplify the formulas, set 
  $$
  \xymatrix{
  F = \prod\limits_{ j = 0}^{n-1} h_{n,j}
  &
  \theta = \beta  F^2
  }
  $$
  Thus we have $H_\beta = (1 + \theta)h_n^2$, and our prior calculation showed when $\beta = 0$, 
 the last equation of  system \eqref{eq2} took the form 
  $$
  (-1)^{n-1} h F = \gamma F^2
  $$
so was solved by $ h = (-1)^{n-1}  \gamma F$.

 For general $\beta$, after factoring out the function $h$ to be solved for, the left-hand
  side of the last equation of  system \eqref{eq2}  takes the form 
   \begin{align*}
  & (-1)^{n-1} (1+\theta) F - h_n^2 \sum\limits_{j=1}^{n-1}\frac1{2^j}\frac{\p\theta}{\p x_j}x_j
  \end{align*}  
%  
%     \begin{align*}
%  & \pm\left(1+\underbrace{\beta \prod\limits_1^{n-1} g_j^{(n-j+1)}}_\theta\right) f_n\cdots f^{n-1}_2q_n^2-g_n\sum\limits_{j=1}^{n-1}\frac1{2^j}\frac{\p\theta}{\p q_j}q_j.
%  \end{align*}  
%
   % Denote, $F:= \prod\limits_1^{n-1} f_j^{(n-j+1)} $ and $G:= f^2\prod\limits_1^{n-1} G_j^{(n-j+1)} $
Thus the equation itself takes the form
\begin{equation}
\label{eq:h-j}
   ((-1)^{n-1} (1+\theta) F - h_n^2 \sum\limits_{j=1}^{n-1}\frac1{2^j}\frac{\p\theta}{\p x_j}x_j)h = \gamma F^2
  \end{equation}

Since  $\theta = \beta F^2$, we have
$$ 
\frac{\p\theta}{\p x_j}=F^2\frac{\p\beta}{\p q_j}+\beta\frac{\p F^2}{\p q_j} = F^2\frac{\p\beta}{\p q_j}+2 F \beta\frac{\p F}{\p q_j}
$$
 and hence $ \frac{\p\theta}{\p x_j}$ is divisible by $F$.
 Thus we can divide equation~\eqref{eq:h-j} by $F$, and after renaming $\gamma$, write equation~\eqref{eq:h-j}  in the form
       \begin{align*}
   (1+ O(x))h = \gamma F
  \end{align*}  
  where $O(x)$ vanishes at the origin. We conclude we can solve the equation by  $h = (1+ O(x))^{-1}\gamma F$.
  
  This completes the proof of Proposition \ref{prop:vf}.  \end{proof}
%   

%%%%%%%%%%%%%%%%%%%
%%%%%%%%%%%%%%%%%%%%
\subsubsection{Proof of Theorem \ref{thm: quad unique}}
In this section, we  use Propositions~\ref{prop:divisibility} and Proposition~\ref{prop:vf}  to complete the  inductive step outlined 
in~\ref{s: inductive step}, and thus, complete the proof of   Theorem \ref{thm: quad unique}.    Let us assume  $n\geq 2$.

Then
  $
  \Upsilon = \Upsilon' \cup \Upsilon_n$ where $ \Upsilon' = \bigcup_{i = 0}^{n-1} {}^n \Gamma_i $, $\Upsilon_n = \pi({}^n \Lambda_n)$. 
  We will implement the following strategy. Suppose for some $0 < k \leq n-1$, we have moved $\Upsilon_n$, while preserving $\Upsilon'$, so that we have the relation of primary tangencies
  $$
  \xymatrix{
  \tau(\Upsilon_n, {}^n \Gamma_{j}) \supset \tau({}^n \Gamma_n, {}^n \Gamma_{j}) & j> k
  }$$
Then using Proposition~\ref{prop:divisibility} and Proposition~\ref{prop:vf}, or alternatively, the cases $n=0, 1$ when respectively $k = n-1, n-2$, we will move  $\Upsilon_n$, while preserving $\Upsilon'$, so that we have the relation of
primary tangencies 
  $$
  \xymatrix{
  \tau(\Upsilon_n, {}^n \Gamma_{j}) \supset \tau({}^n \Gamma_n, {}^n \Gamma_{j}) & j\geq k
  }$$
Proceeding in this way,  we will arrive at $k=0$, where all primary tangencies have been normalized. Then a final 
application of Proposition~\ref{prop:divisibility} and Proposition~\ref{prop:vf} will complete the proof.

    \begin{figure}[h]
\includegraphics[scale=0.4]{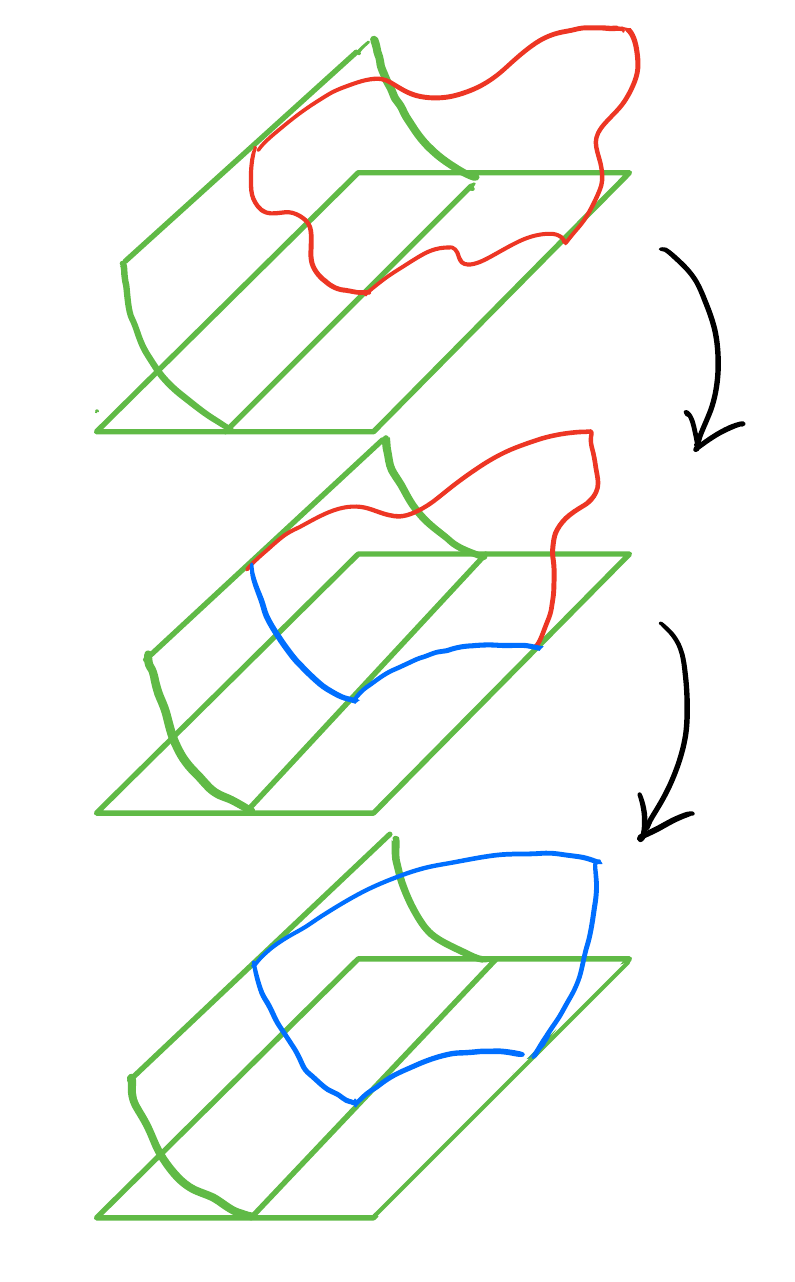}
\caption{The strategy of the proof: inductively normalize tangencies.}
\label{fig:Firstintersections}
\end{figure}
To pursue this argument, we need the following control over primary tangencies.

 \begin{lemma}\label{lm:TT} Fix $0 \leq k < j\leq n-1$. 
 
We have
 $$
 \xymatrix{
 \tau(\tau(\Upsilon_n,{}^n \Gamma_k),\tau({}^n \Gamma_j,{}^n \Gamma_k))\supset 
 \tau(\Upsilon_n,{}^n \Gamma_j)   \cap \tau({}^n \Gamma_j,{}^n \Gamma_k) 
 }$$
  
 Moreover, when $k = n-2$,  the tangency of $\tau(\Upsilon_n,{}^n \Gamma_{n-2})$ and $\tau({}^n \Gamma_{n-1},{}^n \Gamma_{n-2})$ is nondegenerate.
 
 \end{lemma}

\begin{proof} We will assume $k>0$ and leave the case $k=0$ as an exercise. 

Fix a point
$$
\xymatrix{
y\in  \tau(\Upsilon_n,{}^n \Gamma_j)   \cap \tau({}^n \Gamma_j,{}^n \Gamma_k) 
}
$$
In particular $y\in \Upsilon_n$ and so $y = \pi(\tilde y)$ for some $\tilde y\in {}^n \Lambda_n$.  Recall 
${}^n \Lambda_n= \bigcup_{\eps} {}^n \Lambda_n^\eps$ and so $\tilde y \in {}^n \Lambda_n^\eps$, for some $\eps$.

Note $y\in  \tau(\Upsilon_n,{}^n \Gamma_j)$ implies $\tilde y$ is in the closure of the clean codimension one intersection of    ${}^n \Lambda_n,  {}^n \Lambda_j$.

By Corollary~\ref{cor: clean intersect}, this locus intersects ${}^n \Lambda_n^\eps$ precisely along 
${}^n \Lambda^\eps_n \cap  {}^n \Lambda^\eps_j$ and so $\tilde y \in {}^n \Lambda_j^\eps$.

Similarly, note  $y\in  \tau({}^n \Gamma_j,{}^n \Gamma_k)$ implies $\tilde y$ is in the closure of the clean codimension one intersection of    ${}^n \Lambda_j,  {}^n \Lambda_k$. By Corollary~\ref{cor: clean intersect}, this locus intersects ${}^n \Lambda_j^\eps$ precisely along 
${}^n \Lambda^\eps_j \cap  {}^n \Lambda^\eps_k$ and so $\tilde y \in {}^n \Lambda_k^\eps$.

Thus altogether $\tilde y \in {}^n \Lambda^\eps_n \cap  {}^n \Lambda^\eps_j \cap  {}^n \Lambda^\eps_k = 
 ({}^n \Lambda^\eps_n \cap  {}^n \Lambda^\eps_k) \cap  ({}^n \Lambda^\eps_j \cap  {}^n \Lambda^\eps_k)$.
 
By Corollary~\ref{cor: clean intersect},  the intersections 
${}^n \Lambda^\eps_n \cap  {}^n \Lambda^\eps_k$ and  ${}^n \Lambda^\eps_j \cap  {}^n \Lambda^\eps_k$ are closures of clean codimension one intersections, hence their projections lie in the primary tangencies 
$\tau(\Upsilon_n,{}^n \Gamma_k)$ and $\tau({}^n \Gamma_j,{}^n \Gamma_k)$.
Moreover, ${}^n \Lambda^\eps_n \cap  {}^n \Lambda^\eps_k$ and  ${}^n \Lambda^\eps_j \cap  {}^n \Lambda^\eps_k$ intersect along their primary tangency. Since $\pi$ restricted to ${}^n \Lambda_k$ has no critical points, 
the projection of this primary tangency is again a primary tangency.
Hence $y\in \tau(\tau(\Upsilon_n,{}^n \Gamma_k),\tau({}^n \Gamma_j,{}^n \Gamma_k))$, proving the asserted containment.

We leave the nondegeneracy of the case $k = n-2$ to the reader.
\end{proof}

 Now we are ready to inductively normalize the primary tangencies.
 
\begin{lemma}\label{lem:ind}
Fix $0 \leq k<n-1$.

Suppose
  $$
  \xymatrix{
  \tau(\Upsilon_n, {}^n \Gamma_{j}) = \tau({}^n \Gamma_n, {}^n \Gamma_{j}) & j> k
  }$$
 
Then there exists a diffeomorphism $\psi:\R^{n+1} \to \R^{n+1}$ preserving 
 $ \Upsilon' = \bigcup_{i = 0}^{n-1} {}^n \Gamma_i$
%  and 
% $
% T(\Upsilon_n, {}^n \Gamma_{j})$, for all $j>k$, 
   such that
 $$
\xymatrix{
 \tau(\psi(\Upsilon_n),{}^n \Gamma_j)=\tau({}^n\Gamma_n,{}^n \Gamma_j) & j\geq k
 }
 $$
 
 Moreover, when $k \not = n-2$, the diffeomorphism is an isotopy.
\end{lemma}

\begin{proof} We  will assume $k < n-3$. We 
 leave the elementary cases $k = n-2, n-3$ to the reader. They can be  deduced  from the parametric versions of the cases $n=0, 1$ presented in \ref{s: n=0}, \ref{s: n=1} respectively.

Throughout what follows, we use the projection $\R^{n+1} \to \R^n$ to identify ${}^n \Gamma_k = \R^n$.

On the one hand, we have
  $$
  \xymatrix{
  \tau({}^n  \Gamma_j, {}^n \Gamma_k) = \R^k \times {}^{n-k - 1} \Gamma_{j - k -1}
  & k < j < n
  }
  $$
%  and so their union satisfies
%  $$
%  \xymatrix{
%  \bigcup_{j = k+1}^{n-1} {}^n T({}^n  \Gamma_j, {}^n \Gamma_k) = \R^k \times ({}^{n-k - 1} \Gamma \setminus {}^{n-k - 1} \Gamma_{n-k - 1})
%  }
%  $$

On the other hand, by Lemma~\ref{lm:TT} and assumption, we have 
$$
 \xymatrix{
 \tau(\tau(\Upsilon_n,{}^n \Gamma_k),\tau({}^n \Gamma_j,{}^n \Gamma_k))=
 \tau (\Upsilon_n,{}^n \Gamma_j)\cap {}^n \Gamma_k
=
 \tau ({}^n \Gamma_n,{}^n \Gamma_j)\cap {}^n \Gamma_k 
 & k < j < n
 }$$
Hence within ${}^n \Gamma_k = \R^n$, the loci $\tau(\Upsilon_n,{}^n \Gamma_k)$ 
and  $\tau({}^n \Gamma_n,{}^n \Gamma_k)$  have the same tangencies with 
  $$
  \xymatrix{
  \tau({}^n  \Gamma_j, {}^n \Gamma_k) = \R^k \times {}^{n-k - 1} \Gamma_{j - k -1}
  & k < j < n
  }
  $$

 Thus Proposition~\ref{prop:divisibility} and Proposition \ref{prop:vf} provide a time-dependent vector field
 of the form 
 $$
 \xymatrix{
 v_t =h_t \sum\limits_{i=k+1}^{n-1}\frac1{2^i}x_i  \partial_{x_i}
 }
 $$
 generating an isotopy
$\varphi:\R^{n-k} \to \R^{n-k}$ satisfying 
$$
\xymatrix{
\varphi(\tau(\Upsilon_n,{}^n \Gamma_k)) = \tau({}^n \Gamma_n,{}^n \Gamma_k)
}
$$
In addition,  the function $h_t$, hence vector field $v_t$, is divisible by  the
 product
$
\prod_{j = k+1}^{n- 1} h_{n, j},
$
and thus $\varphi$ preserves its zero-locus.

Let us complete $v_t$ to the vector  field
 $$
 \xymatrix{
 V_t=h_t \sum\limits_{i=0}^{n-1}\frac1{2^i}x_i \partial_{x_i} 
 }$$ 
 and consider the isotopy
$\psi:\R^{n+1} \to \R^{n+1}$ generated by $V_t$.

 Then $\psi$  satisfies
  $$
\xymatrix{
\psi(\tau(\Upsilon_n,{}^n \Gamma_k)) = \tau({}^n \Gamma_n,{}^n \Gamma_k)
}
$$
It also preserves ${}^n \Gamma_i$, for $0 \leq i \leq  n-1$, as well as 
  $
  \tau(\Upsilon_n, {}^n \Gamma_{j}) = \tau({}^n \Gamma_n, {}^n \Gamma_{j})
  $,
  for $j > k$.
 In addition,  it preserves
 $$
  \xymatrix{
  \tau({}^n  \Gamma_j, {}^n \Gamma_k) = \R^k \times {}^{n-k - 1} \Gamma_{j - k -1}
  & k < j < n
  }
  $$
  since this is the zero-locus of $h_{n, j}$.  
\end{proof}
 
 Finally, let us use the lemma to complete the inductive step of the proof of Theorem~\ref{thm: quad unique} as outlined above.
   Suppose for some $0 < k \leq n-1$, we have moved $\Upsilon_n$, while preserving $\Upsilon'$, so that we have the sought-after primary tangencies 
  $$
  \xymatrix{
\tau(\Upsilon_n, {}^n \Gamma_{j}) = \tau({}^n \Gamma_n, {}^n \Gamma_{j}) & j> k
  }$$
Then using Lemma~\ref{lem:ind},  we can move  $\Upsilon_n$, while preserving $\Upsilon'$, so that we have the sought-after primary tangencies 
  $$
  \xymatrix{
  \tau(\Upsilon_n, {}^n \Gamma_{j}) = \tau({}^n \Gamma_n, {}^n \Gamma_{j}) & j\geq k
  }$$
Proceeding in this way,  we arrive at $k=0$, where all primary tangencies have been normalized. Now a final 
application of Proposition~\ref{prop:divisibility} and Proposition~\ref{prop:vf} move  $\Upsilon_n$ to ${}^n\Gamma_n$, 
while preserving $\Upsilon'$, and thus complete the proof of Theorem \ref{thm: quad unique}.
  
 \subsection{Conclusion of the proof}\label{sec:proof-cone-arb}
We are now ready to prove Proposition \ref{prop:unique-unsigned}. As a consequence we establish  Theorem \ref{thm:unique-signed}, and since all the above also holds parametrically this also establishes the parametric version Theorem \ref{thm:unique-signed-parametric}.

 \begin{proof}[Proof of Proposition \ref{prop:unique-unsigned}]
 Take  any point $\lambda$ in the front $H:= \pi(\Lambda)$  and let $\pi^{-1}(\lambda)=\{\lambda_1,\dots,\lambda_k\}$. Let    $\Lambda_1,\dots, \Lambda_k$ be germs  of $\Lambda$ at these  points  of arboreal types $(\sT_j,n)$, $n(\sT_j)=n_j$. 
 %Suppose the theorem already proven. 
 We need to show that the germ of the front $H$ at $\lambda$ is diffeomorphic to the germ of  a model  front $H_\sT$, where $\sT$ is a signed rooted tree obtained from $\bigsqcup T_j$ by adding the root $\rho$ and adjoining it to the roots $\rho_j$ of the trees $T_j$ by edges $[\rho\rho_j]$.  The signs of all edges of  the trees $T_j$ are preserved, while previously unsigned edges
 $\rho_j\alpha$  get a  sign $\eps(\nu,L,\alpha)$, see \eqref{eq:arb-sign}.
 
%%%%%
We proceed by induction on the number of vertices in the signed rooted tree $\sT = (T, \rho, \eps)$.

The base case of a $(\sA_1, m)$-front $H \subset  \R^m$ is the same geometry as appearing in \ref{s: n=0}:
any graphical hypersurface $H \subset  \R \times \R^{m-1}$ is isotopic to the germ of the zero-graph $ \{0\} \times \R^{m-1}$.

For the inductive step, fix  a   rooted tree $\sT = (T, \rho, \eps)$, and as usual set $n = |n(\sT)|$. Consider  a $(\sT, m)$-front $H \subset \R^{m}$, with by necessity $m\geq n$.

Fix  a leaf vertex $\beta \in \ell(\sT)$, which always exists as long as $\sT\not = \sA_1$. Consider the smaller signed rooted tree $\sT' = \sT\setminus \beta$, and 
 the corresponding 
$(\sT', m)$-front $H' = H \setminus \mathring H[\beta] \subset  \R^m$, where $\mathring H[\beta] \subset H$ is the interior of the smooth piece indexed by $\beta$. 
By induction,
 we may assume
$$
\xymatrix{
H' = H_{\sT'} \times \R^{m-n+1} \subset    \R^m
}
$$
Thus it remains to normalize the smooth piece $H[\beta]$.

Let $\sA_\beta = (A_\beta, \rho, \eps_\beta)$ be the linear signed rooted subtree of  $\sT = (T, \rho, \eps)$ with vertices  
$v(A_\beta) = \{ \alpha \in v(T) \, |\, \alpha\leq \beta\}$. Set  $d=  v(\sT) \setminus v(\sA_\beta) =  n(\sT) \setminus n(\sA_\beta)$ to be the complementary vertices. 

Consider the $(\sA_\beta, m)$-front $K \subset H$ given by the union $K = \bigcup_{\alpha \in n(\sA_\beta)} K[\alpha]$ of the smooth pieces of $H \subset \R^m$ indexed by $\alpha \in n(\sA_\beta)$. 
Note for $\sA'_\beta = \sA_\beta \cap \sT'$, and $K'  = K \cap H'$, we  already have 
$$
\xymatrix{
K' = H_{\sA_\beta'} \times \R^{m-n+1+d} \subset  \R^m
}
$$
and seek to normalize the smooth piece $K[\beta] = H[\beta]$.

Now we can apply Theorem \ref{thm: quad unique} to normalize $K[\beta]$ viewed as the final smooth piece of
$K$. More specifically, we can apply Theorem \ref{thm: quad unique} to normalize $K[\beta]$ while preserving $K'$ and viewing the complementary directions $\R^{m-n+1+d}$ as parameters, see Figure \ref{fig:parametric}. This insures we preserve $H'$ and hence do not disturb its already arranged normalization.

%%%%%%%%
This concludes the proof of Proposition \ref{prop:unique-unsigned}. \end{proof}

       \begin{figure}[h]
\includegraphics[scale=0.4]{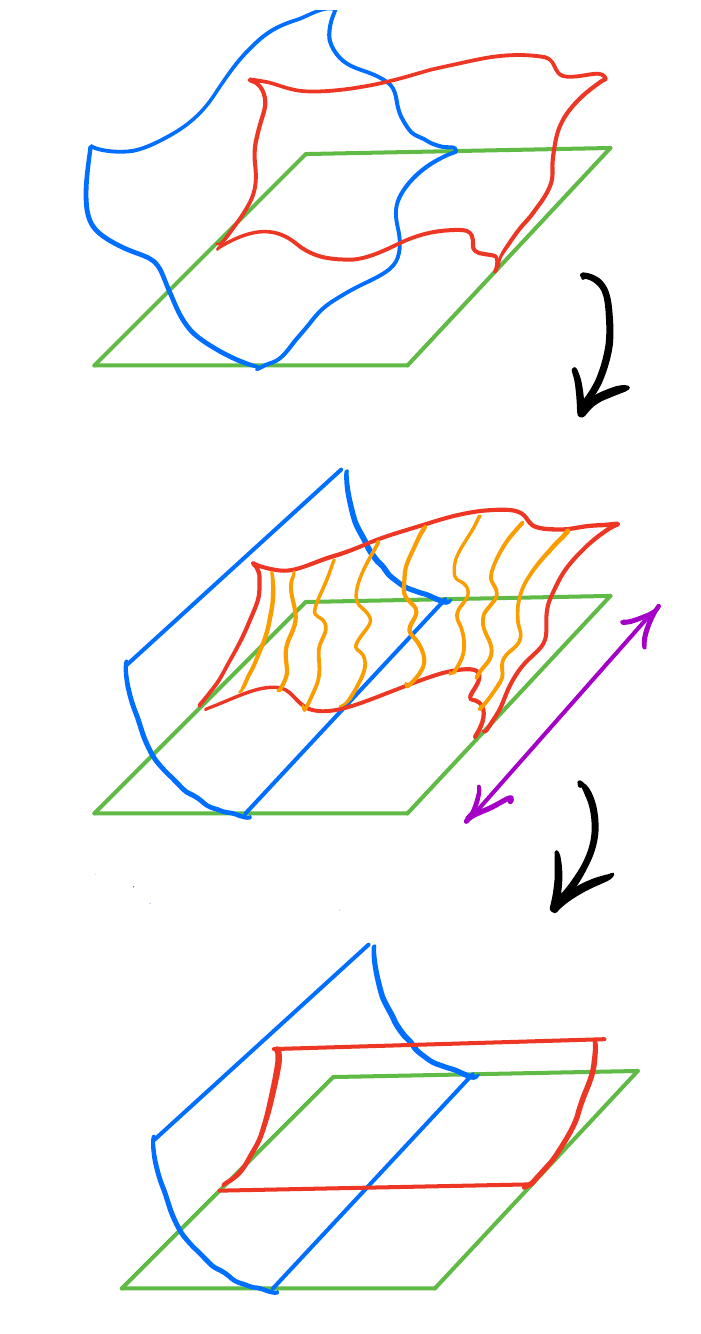}
\caption{Treating the complementary directions as parameters.}
\label{fig:parametric}
\end{figure}
%%%%%%%%%%%%%%%%%%%%%%%%%%%%%%%%%%%%%%%%%%%%%%%%%%%%%%%%%%%%%%%%%%%%%%%%%%%%%%%%%%%%%

 \end{document}